\documentclass{amsart}
\usepackage{amssymb,amsmath,amsthm,enumerate,amscd,graphicx,color}
\usepackage[usenames,dvipsnames,svgnames,table]{xcolor}
\usepackage{pdfsync}
 \usepackage[colorlinks=true, pdfstartview=FitV, linkcolor=blue, citecolor=blue, urlcolor=blue]{hyperref}
\newcommand{\thmref}[1]{Theorem~\ref{#1}}

\newcommand{\lemref}[1]{Lemma~\ref{#1}}
\newcommand{\eqnref}[1]{~(\ref{#1})}

\newcommand{\germ}{\mathfrak}
\newtheorem{thm}{Theorem}[section]
\newtheorem{lem}[thm]{Lemma}

\theoremstyle{definition}
\newtheorem{cor}[thm]{Corollary}

\newtheorem{prop}[thm]{Proposition}


\theoremstyle{rem}
\numberwithin{equation}{section}

\usepackage{amssymb}
\begin{document}
\title{The three point gauge algebra $\mathcal V\ltimes \mathfrak{sl}(2, \mathcal R)  \oplus\left( \Omega_{\mathcal R}/d{\mathcal R}\right)$  and an action on a Fock space.}
\author{Ben Cox}
\begin{abstract}  The three point current algebra $\mathfrak{sl}(2,\mathcal  R)$ where $\mathcal R=\mathbb C[t,t^{-1},u\,|\,u^2=t^2+4t ]$ and three-point Virasoro algebra both act on a previously constructed Fock space. In this paper we prove that the semi-direct product, i.e. the  gauge algebra acts on the Fock space as well. 
\end{abstract}
\address{Department of Mathematics \\
The College of Charleston \\
66 George Street  \\
Charleston SC 29424, USA}\email{coxbl@cofc.edu}
\urladdr{http://coxbl.people.cofc.edu/papers/preprints.html}
\author{Elizabeth Jurisich}
\address{Department of Mathematics,
The College of Charleston,
Charleston SC 29424}
\email{jurisiche@cofc.edu}
\author{Renato A. Martins}
\address{Institute of Science and Tecnology,
Federal University of Sao Paulo,
Sao Jose dos Campos SP 12247014, Brazil}
\email{martins.renato@unifesp.br}

\maketitle
\section{Introduction}  

It is well known from the work of C. Kassel and J.L. Loday (see \cite{MR694130}, and \cite{MR772062}) that if $R$ is a commutative algebra and $\mathfrak g $ is a simple finite dimensional Lie algebra, both defined over the complex numbers, then the universal central extension $\hat{{\mathfrak g}}$ of ${\mathfrak g}\otimes R$ is the vector space $\left({\mathfrak g}\otimes R\right)\oplus \Omega_R^1/dR$ where $\Omega_R^1/dR$ is the space of K\"ahler differentials modulo exact forms (see \cite{MR772062}).  The vector space $\hat{{\mathfrak g}}$ is made into a Lie algebra by defining
$$
[x\otimes f,y\otimes g]:=[xy]\otimes fg+(x,y)\overline{fdg},\quad [x\otimes f,\omega]=0
$$
for $x,y\in\mathfrak g$, $f,g\in R$,  $\omega\in \Omega_R^1/dR$ and $(-,-)$ denotes the Killing form on $\mathfrak g$.  Here $\overline{a}$ denotes the image of $a\in\Omega^1_R$ in the quotient $\Omega^1_R/dR$.      A somewhat vague but natural question comes to mind
is whether there exists free field or Wakimoto type realizations of these algebras.  It is well known from the work of M. Wakimoto and B. Feigin and E. Frenkel what the answer is when $R$ is the ring of Laurent polynomials in one variable (see \cite{W} and \cite{MR92f:17026}).  We review our previous work in sections below about a realization in the setting where $\mathfrak g=\mathfrak{sl}(2,\mathbb C)$ and $R=\mathbb C[t,t^{-1},u|u^2=t^2+4t]$ is the three point algebra (see \cite{MR3245847}).  This is an example of a Krichever-Novikov algebra.

   In Kazhdan and Luszig's explicit study of the tensor structure of modules for affine Lie algebras (see \cite{MR1186962} and \cite{MR1104840}) the ring of functions regular everywhere except at a finite number of points appears naturally.  In particular in the monograph  \cite[Ch. 12]{MR1849359} algebras of the form $\oplus _{i=1}^n\mathfrak g((t-x_i))\oplus\mathbb Cc$ appear in the description of the conformal blocks.  These contain the {\it $n$-point algebras} $ \mathfrak g\otimes \mathbb C[(t-x_1)^{-1},\dots, (t-x_N)^{-1}]\oplus\mathbb Cc$ modulo part of the center $\Omega_R/dR$.   M. Bremner explicitly described the universal central extension of such an algebra in \cite{MR1261553}. This too is an example of a Krichever-Novikov algebra.
   
This is a Krichiver-Novikov algebra of the form $\text{Der}(R)$ where $R$ is the ring of meromorphic functions on a Riemann surface with a finite number of points removed. 
 But since  $R=\mathbb C[t,t^{-1},u|u^2=t^2+4t]$, we call the algebra $\text{Der}(R)$ and its central extensions, when  $R=\mathbb C[t,t^{-1},u|u^2=t^2+4t]$, the three point Virasoro algebra.

  The terminology for the {\it gauge algebra} comes from the book \cite{MR1629472}.  It is by definition the semi-direct product of $\mathfrak g\otimes R$ and $\text{Der}(R)$. We use the above names because this terminology is precise and one can use it to distinguish which algebra one is writing about from amongst the plethora of other wonderful Krichever-Novikov algebras.  \color{black}   
   
   \color{black}
   
   Consider now the Riemann sphere $\mathbb  C\cup\{\infty\}$ with coordinate function $s$ and fix three distinct points $a_1,a_2,a_3$ on this Riemann sphere.    Let $R$ denote the ring of rational functions with poles only in the set $\{a_1,a_2,a_3\}$.  It is known that the automorphism group $PGL_2(\mathbb C)$ of $\mathbb C(s)$ is simply 3-transitive and $R$ is a subring of $\mathbb C(s)$, so that $R$ is isomorphic to the ring of rational functions with poles at $\{\infty,0,1,a\}$.  Motivated by this isomorphism one sets $a=a_4$ and here the {\it $4$-point ring} is $R=R_a=\mathbb C[s,s^{-1},(s-1)^{-1},(s-a)^{-1}]$ where $a\in\mathbb C\backslash\{0,1\}$.    Let $S:=S_b=\mathbb C[t,t^{-1},u]$ where $u^2=t^2-2bt+1$ with $b$ a complex number not equal to $\pm 1$.  Then M. Bremner has shown us that $R_a\cong S_b$. 
As the later, being $\mathbb Z_2$-graded, is a cousin to super Lie algebras, and  is thus more immediately amendable to the theatrics of conformal field theory.  Moreover Bremner has given an explicit  description of the universal central extension of $\mathfrak g\otimes R$, in terms of ultraspherical (Gegenbauer) polynomials where $R$ is the four point algebra (see \cite{MR1249871}).      In \cite{MR2373448} the first author gave a realization for the four point algebra where the center acts nontrivially. 

In his study of the elliptic affine Lie algebras, $\mathfrak{sl}(2, R)  \oplus\left( \Omega_R/dR\right)$ where $R=\mathbb C[x,x^{-1},y\,|\,y^2=4x^3-g_2x-g_3]$, M. Bremner has also explicitly described the universal central extension of this algebra in terms of Pollaczek polynomials (see  \cite{MR1303073}).   Essentially the same algebras appear in recent work of A. Fialowski and M. Schlichenmaier \cite{FailS} and \cite{MR2183958}.  Together with Andr\'e Bueno and Vyacheslav Futorny, the first author described free field type realizations of the elliptic Lie algebra where $R=\mathbb C[t,t^{-1},u\,|, u^2=t^3-2bt^2-t]$, $b\neq \pm 1$ (see \cite{MR2541818}).

 Below we look at the three point algebra case where $R$ denotes the ring of rational functions with poles only in the set $\{a_1,a_2,a_3\}$. This algebra is isomorphic to $\mathbb C[s,s^{-1},(s-1)^{-1}]$. M. Schlichenmaier has a slightly different description of the three point algebra as $\mathbb C[(z^2-a^2)^k,z(z^2-a^2)^k\,|\, k\in\mathbb Z]$ where $a\neq 0$ (see \cite{MR2058804}). 
We show that $R\cong \mathbb C[t,t^{-1},u\,|\,u^2=t^2+4t]$ and thus looks more like $S_b$ above.
 The first two authors (see \cite{MR3245847}, Theorem 5.1)  provides a natural free field realization in terms of a $\beta$-$\gamma$-system and the oscillator algebra of the three point affine Lie algebra when $\mathfrak g=\mathfrak{sl}(2,\mathbb C)$.   Just as in the case of intermediate Wakimoto modules defined in
\cite{ MR2271362}, there are two different realizations depending on two different normal orderings.  Besides M. Bermner's article mentioned above, other work on the universal central extension of $3$-point algebras can be found in \cite{MR2286073}. Previous related work on highest weight modules of $\mathfrak{sl}(2,R)$ can be found in H. P. Jakobsen and V. Kac \cite{JK}.

The three point algebra is perhaps the simplest non-trivial example of a Krichever-Novikov algebra beyond an affine Kac-Moody algebra 
 (see \cite{MR902293}, \cite{MR925072}, \cite{MR998426}).  A fair amount of interesting and 
fundamental work has be done by Krichever, Novikov, Schlichenmaier, and Sheinman on the representation theory of the Krichever-Novikov algebras. 
 In particular Wess-Zumino-Witten-Novikov theory and analogues of the Knizhnik-Zamolodchikov equations are developed for  these algebras 
(see the survey article \cite{MR2152962}, and for example \cite{MR1706819}, \cite{MR1706819},\cite{MR2072650},\cite{MR2058804},\cite{MR1989644}, and \cite{MR1666274}). 

The initial motivation for the use of Wakimoto's realization was to prove a conjecture of V. Kac and D. Kazhdan on the character of certain irreducible representations of affine Kac-Moody algebras at the critical level (see \cite{W} and \cite{MR2146349}). Another motivation for constructing free field realizations is that they are used to provide integral solutions to the KZ-equations (see for example \cite{MR1077959} and \cite{MR1629472} and their references).  A third is that they are used to help in determining the center of a certain completion of the enveloping algebra of an affine Lie algebra at the critical level which is an important ingredient in the geometric Langland's correspondence \cite{MR2332156}.  Yet a fourth is that free field realizations of an affine Lie algebra appear naturally in the context of the generalized AKNS hierarchies \cite{MR1729358}.

The outline of the paper is as follows.  In section two we pin down the module structure of how the automorphism group of $R$ acts on the center $\Omega_R/dR$.  The third section is background material on the $\beta$-$\gamma$ system and the $3$-point Heisenberg algebra and part of  their respective representation theories. In sections four and five we recall material form our previous work on the $3$-point Virasoro algebra and the $3$-point current algebra and their respective representation theories.  In section six we put these representations together to get conditions on the action of  the gauge algebra on the corresponding Fock space so as to get a well defined representation.

Although we couldn't produce representations for which all of the center acts nontrivially, we will explore in future work whether there is an ambient gerbe lurking in the background. 

The authors would like to thank Murray Bremner for bringing to their attention the need for more research on Krichever-Novikov algrabras and the first two authors would also like to thank the Mathematics Department at the College of Charleston for a summer Research and Development Grant which supported this work.   The first author is partially supported by a collaboration grant from the Simons Foundation (\#319261) and  he would like to thank them for their generous support. The third author was supported by FAPESP grant (2014/09310-5).

\section{The $3$-point ring and its group of automorphisms.}
The three point algebra has at least four incarnations. 
\subsection{Three point algebras}  Fix $0\neq a\in\mathbb C$.
Let \begin{align*}
\mathcal S&:=\mathbb C[s,s^{-1},(s-1)^{-1} ] ,\\
 \mathcal R&:=\mathbb C[t,t^{-1},u\,|\, u^2=t^2+4t], \\
 \mathcal A&:=\mathcal A_a=\mathbb C[(z^2-a^2)^k,z(z^2-a^2)^k\,|\,k\in\mathbb Z].
 \end{align*}
 Note that Bremner introduced the ring $\mathcal S$ and Schichenmaier introduced $\mathcal A$ (see \cite{MR2058804}).  Variants of $\mathcal R$ were introduced by Bremner for elliptic and $3$-point algebras.  The last incarnation of the three point algebra, which we will not review, is due to G. Benkart and P. Terwilliger \cite{MR2286073}\color{black}.  This later one is a bit mysterious to us.
\begin{lem} 
\begin{enumerate}
\item There is an isomorphism of rings $
\phi^{-1}:\mathcal R\to \mathcal S$ whereby $t\mapsto s^{-1}(s-1)^2=s^{-1}-2+s$, and $u\mapsto s-s^{-1}$.
\item The rings $\mathcal R$ and $\mathcal A$ are isomorphic where the isomorphism is implemented by $z\mapsto a(1+t+u)$.
\end{enumerate}
\end{lem}
\begin{proof} 
(1). Let $\bar f: \mathbb C[t,u] \rightarrow \mathcal S$ be the ring homomorphism defined $\bar f(t)=s^{-1}(s-1)^2=s-2+s^{-1}$, $\bar f(u)= s-s^{-1}$.

 We first check that 
$$
\bar f(u^2-(t^2+4t))=(s-s^{-1})^2-(s-2+s^{-1})^2-4(s-2+s^{-1})=0
$$ 
and $\bar f(t)=s^{-1}(s-1)^2$ is invertible in $\mathcal S$.  Hence the map $\bar f$ descends to a well defined ring homomorphism $f:\mathcal R\to \mathcal S$.  To show that it is onto we essentially solve for $s$ and $s^{-1}$ in terms of $t$ and $u$.  The inverse ring homomorphism of $f$ is $\phi:\mathcal S\to \mathcal R$ given by 
$$
\phi(s)=\frac{t+2+u}{2},\quad \phi(s^{-1})=\frac{t+2-u}{2}.
$$
In particular $\displaystyle{\phi((s-1)^{-1})=\frac{t^{-1}u-1}{2}}$.

For part (2) observe that  $\mathcal A=\mathbb C[z,(z-a)^{-1},(z+a)^{-1}]$ which after setting $z=2as-a$ we get $\mathcal A=\mathbb C[s,s^{-1},(s-1)^{-1}]$.    Thus an isomorphism between $\mathcal A$ and $\mathcal R$ is implemented by the assignment $z\mapsto a(1+t+u)$, $(z+a)^{-1}\mapsto\displaystyle{\frac{t+2-u}{4a}}$ and $(z-a)^{-1}\mapsto   \displaystyle{ \frac{t^{-1}u-1}{4a}}$.  \end{proof}

The group of automorphisms of $\mathcal S$ was described in \cite{MR3211093} among other automorphism groups.  In this case the group of algebraic automorphisms $\text{Aut}(\mathcal S)$ is the dihedral group of order $6$ with elements $\iota, \varphi, \pi_i$, $i=1,2,3,4$ where $\iota$ is the identity map and 
\begin{gather*}
\varphi(s)=1-s,\quad \pi_1(s)=\frac{1}{s},\quad \pi_2(s)=\frac{1}{1-s},\\  
\pi_3(s)=\frac{s}{s-1},\quad 
\pi_4(s)=\frac{s-1}{s}.\end{gather*}
These are all M\"obius transformations with $\varphi$ of order 2 and $\pi_2$ of order 3 with 
$
\phi\pi_2\phi=\pi_2^{-1}$.
\begin{cor}  Setting $\psi=\phi\varphi\phi^{-1}$ and $\tau_2=\phi\pi_2\phi^{-1}$ one obtains, the group $\text{Aut}(\mathcal R)$ is the dihedral group $D_3$ of order $6$ where \begin{gather}
 \psi(t)=\frac{-t^{-1}u-3-t-u}{2},\quad \psi(u)=\frac{t^{-1}u-1-t-u}{2},\\
 \tau_2(t)=\frac{-t^{-1}u-3-t-u}{2} ,\quad  \tau_2(u) =\frac{-t^{-1}u+t+1+u}{2},  
 \end{gather}
 \end{cor}
\begin{proof}  The isomorphism $\phi^{-1}:\mathcal R\to \mathcal S$ induces the automorphisms above where $\psi=\phi\circ\varphi\circ\phi^{-1}$, and $\tau_i=\phi\circ\pi_i\circ\phi^{-1}$. 
 \begin{align*}
 \psi(t)&=\phi\circ\varphi\circ\phi^{-1}(t)=\phi(\varphi(s^{-1}-2+s))=\phi\left(\frac{1}{1-s}-1-s\right) \\
 &=\frac{-t^{-1}u+1}{2}-1-\left(\frac{t+2+u}{2}\right)=\frac{-t^{-1}u-3-t-u}{2} \\  \\
 \psi(u)&=\phi\circ\varphi\circ\phi^{-1}(u)=\phi(\varphi(s-s^{-1}))=\phi\left(1-s-\frac{1}{1-s}\right) \\
 &=\frac{t^{-1}u-1}{2}+1- \frac{t+2+u}{2}=\frac{t^{-1}u-1-t-u}{2} 
   \end{align*}
 and 
 \begin{align*}
 \tau_2(t)&=\phi\circ\pi_2(s-2+s^{-1})=\phi^{-1}\left(\frac{1}{1-s}-1-s\right)\\  
 &=\frac{-t^{-1}u-3-t-u}{2} \\ 
 \tau_2(u)&=\phi\circ\pi_2(s-s^{-1})   \\ 
 &=\phi^{-1}\left(\frac{1}{1-s}-(1-s)\right)=\frac{-t^{-1}u+1}{2}-1+\left(\frac{t+2+u}{2}\right) \\
& =\frac{-t^{-1}u+t+1+u}{2}   \\ 
 \end{align*}
\end{proof}
\begin{cor}  The following gives the action of the group of automorphism of $\mathcal S$ on the center $\Omega_{\mathcal S}/d\mathcal S$.  The matrix for $\phi$ on the basis $\{\overline{s^{-1}ds},\overline{(s-1)^{-1}\,ds}\}$ is 
$$
\begin{pmatrix} 0 & 1 \\ 1 & 0\end{pmatrix}
$$

and for 
 $\pi_2$ is
$$
\begin{pmatrix} 0& 1 \\ -1 & -1 \end{pmatrix}
$$
and $\Omega_{\mathcal S}/d\mathcal S$ is irreducible as an $\text{Aut}(\mathcal S)\cong D_3$ module.
\end{cor}
\begin{proof}
We calculate
\begin{align*}
\overline{\varphi(s^{-1})\otimes \varphi(s)}&= \overline{(1-s)^{-1}\otimes(1-s)}= \overline{(s-1)^{-1}\otimes s} \\
\overline{\varphi(s-1)^{-1}\otimes \varphi(s)}&= \overline{-s^{-1}\otimes (1-s)} =\overline{ s^{-1}\otimes s} \\
\end{align*}
Thus the matrix for $\phi$ on the basis $\{\overline{s^{-1}ds},\overline{(s-1)^{-1}\,ds}\}$ is 
$$
\begin{pmatrix} 0 & 1 \\ 1 & 0\end{pmatrix}.
$$
Then $\Omega_{\mathcal S}/d\mathcal S$ is a two dimensional irreducible representation for $\text{Aut}(\mathcal S)\cong D_3$.
Moreover
\begin{align*}
\overline{\pi_2(s^{-1})\otimes \pi_2(s)}&= \overline{ (1-s)\otimes  (1-s)^{-1}}= \overline{(s-1)^{-1}\otimes s} \\
\overline{\pi_2(s-1)^{-1}\otimes \pi_2(s)}&=- \overline{s^{-1}(s-1)\otimes (1-s)^{-1}} =\overline{ s^{-1}(s-1)^{-1}\otimes s} \\
&=-\overline{ s^{-1}\otimes s}-\overline{ (s-1)^{-1}\otimes s} \\
\end{align*}
Then the matrix for $\pi_2$ is
$$
\begin{pmatrix} 0& 1 \\ -1 & -1 \end{pmatrix}.
$$

We want to know whether $\Omega_{\mathcal S}/d\mathcal S$ is an irreducible representation for $\text{Aut}(\mathcal S)\cong D_3$.
To that end we calculate the character of this action $\chi_{\Omega_{\mathcal S}/d\mathcal S}$.  The conjugacy classes of $\text{Aut}(\mathcal S)$ are 
$$
\{1\},\quad \{\phi,\phi\pi_2,\phi\pi_2^2\},\quad  \{\pi_2,\pi_2^2\}.
$$
and the character $\chi_{\Omega_{\mathcal S}/d\mathcal S}$, then has values 
$$
2,\enspace 0,\enspace -1
$$
respectively on these conjugacy classes.  Now this is the same as the character for the unique 2-dimensional irreducible representation for $D_3$, so $\Omega_{\mathcal S}/d\mathcal S$ is irreducible (see \cite{MR0450380}). 
\end{proof}

\subsection{The Universal Central Extension of the Current Algebra $\mathfrak g\otimes  R$.}
 Let $R$ be a commutative algebra defined over $\mathbb C$.
Consider the left $R$-module  $F=R\otimes R$ with left action given by $f( g\otimes h ) = f g\otimes h$ for $f,g,h\in R$ and let $K$  be the submodule generated by the elements $1\otimes fg  -f \otimes g -g\otimes f$.   Then $\Omega_R^1=F/K$ is the module of {\it K\"ahler differentials}.  The element $f\otimes g+K$ is traditionally denoted by $fdg$.  The canonical map $d:R\to \Omega_R^1$ is given by $df  = 1\otimes f  + K$.  The {\it exact differentials} are the elements of the subspace $dR$.  The coset  of $fdg$  modulo $dR$ is denoted by $\overline{fdg}$.  As C. Kassel showed the universal central extension of the current algebra 
$\mathfrak g\otimes R$
 where $\mathfrak g$ is a simple finite dimensional Lie algebra defined over $\mathbb C$, is the vector space 
 \begin{equation}\hat{\mathfrak g}=(\mathfrak g\otimes R)\oplus \Omega_R^1/dR  \label{3pointDef} \end{equation}
 with Lie bracket given by
\begin{equation*}
[x\otimes f,Y\otimes g]=[xy]\otimes fg+(x,y)\overline{fdg},  [x\otimes f,\omega]=0,  [\omega,\omega']=0,
\end{equation*}
  where $x,y\in\mathfrak g$, and $\omega,\omega'\in \Omega_R^1/dR$ and $(x,y)$  denotes the Killing  form  on $\mathfrak g$.  
  
There are at least four incarnations of the three point algebras, three of which are defined as $\mathfrak g\otimes R\oplus \Omega_R/dR$ where $R=\mathcal S, \mathcal R,\mathcal A$ given above.  The forth incarnation appears in the work of G. Benkart and P. Terwilliger given in terms of the tetrahedron algebra (see \cite{MR2286073}). 
We will only work with $R=\mathcal R$.

\begin{prop}[\cite{MR1261553}, see also \cite{MR1249871}]\label{uce}  Let $\mathcal R$ be as above, the ring  $\Omega_\mathcal R^1/d\mathcal R$ has the following properties:
\begin{enumerate}
\item The set 
$
\{\omega_0:=\overline{t^{-1} dt},\enspace \omega_1:=\overline{t^{-1}u\,dt}\}
$
 is a basis of $\Omega_\mathcal R^1/d\mathcal R$.
\item $\overline{t^k\,dt^l} =-k\delta_{l,-k}\omega_0 $
\item  $ \overline{t^ku\,d(t^lu)}=\left((l+1)\delta_{k+l,-2}+(4l +2)\delta_{k+l,-1}\right)\omega_0 $ 
\item $ \overline{t^k\,d(t^lu)}=  \mu_{k,l}\omega_1 $ where $\mu_{k,l}=k\frac{(-1)^{k+l+1}2^{k+l}(2(k+l)-1)!!}{(k+l+1)!}$. 

 \end{enumerate}
 where for an integer $n$ one defines recursively 
 \begin{equation}
 n!!:=\begin{cases}
 1 & \quad \text{ if } n=0\text{ or }n=1 \\
 n\times (n-2)!! & \quad \text{ if } n\geq 2.
 \end{cases}
 \end{equation}
\end{prop}
\begin{proof}  The proof follows almost exactly along the lines of \cite{MR1249871} and \cite{MR1261553} and would be omitted if not for the fact that we need some of the formulae that appear in the proof.  As $\mathcal R$ has basis $\{t^k,t^lu\,|\,k,l\in\mathbb Z\}$, $\Omega_{\mathcal R}$ has a spanning set given by the image in the quotient $F/K$ of the tensor product of these basis elements.  We know $\frac{1}{2}u\,d(u^2)=u^2\,du$.    Next we observe that since $u^2=t^2+4t$ one has in $\Omega_{\mathcal R}$,
\begin{align*}
u(t+2)\,dt=\frac{1}{2}u(2t+4)\,dt=\frac{1}{2}u\,d(u^2)=u^2\,du=(t^2+4t)\,du
\end{align*}
and after multiplying this on the left by $t^k$ we get
\begin{align}\label{recursion}
(t^{k+1}+2t^k)u\,dt-(t^{k+2}+4t^{k+1})\,du=0
\end{align}
in $\Omega_{\mathcal R}$.
Now
\begin{align}
d(t^k)&=kt^{k-1}\,dt, \notag \\
d(t^ku)&=t^k\enspace du+ kt^{k-1}u\,dt, \label{eqn1}
\end{align}
so that in the quotient $\Omega_{\mathcal R}/d\mathcal R$ we have by \eqnref{recursion} followed by \eqnref{eqn1}
\begin{align} 
0&\equiv (t^{k+1}+2t^k)u\,dt-(t^{k+2}+4t^{k+1})\,du\notag \\
&\equiv (t^{k+1}+2t^k)u\,dt-(-(k+2)t^{k+1} -4(k+1)t^{k} )u\,dt \notag\\
&\equiv  \left((k+3)t^{k+1}+(4k+6)t^k\right)u\,dt. \notag
\end{align}
Then 
\begin{equation}
t^k u\,dt\equiv -\frac{(k+3)}{(4k+6)}t^{k+1}u\, dt \mod d\mathcal R\label{eqn2}
\end{equation}
so that 
$$
t^{-k} u\,dt\equiv t^{-3} u\,dt\equiv 0\mod d\mathcal R,\quad k\geq 3,
$$
and 
\begin{equation}
t^{k+1}u\,dt\equiv - \frac{(4k+6)}{(k+3)}t^k u\,dt\mod   d\mathcal R, \notag
\end{equation}
so that $t^{k}u\,dt$ can be written in terms of $t^{-1}u\,dt$ for $k\geq -2$ modulo $d\mathcal R$. 
Thus $\Omega_{\mathcal R}$ is spanned as a left $\mathcal R$-module by $dt $ and $du$, furthermore 
\begin{align}
t^{k-1}\,dt&\equiv \frac{1}{k}d(t^k)\equiv 0\mod d\mathcal R, \text{ for }k\neq 0\label{eqn3}
\end{align}
By equations \eqnref{eqn1}, \eqnref{eqn2}, and \eqnref{eqn3} we have  $\Omega_{\mathcal R}/d\mathcal R$ is spanned by  $\{\overline{t^{-1}\,dt},\overline{t^{-1}u\,dt}\}$.
We know by the Riemann-Roch Theorem that the dimension of this space of K\"ahler differentials modulo exact forms on the sphere with three punctures has dimension 2 (see \cite{MR1261553}).  
\end{proof}


The following result is probably well known, but we couldn't find a source for it. 
\begin{prop}  Let $R$ be a commutative algebra over $\mathbb C$.  Define an action $\text{Der}(R)$ on $R\otimes R$ via the assignment
$$
D(a\otimes b):=D(a)\otimes b+a\otimes D(b)
$$
for $a,b\in R$.   This action descends to an action of $\text{Der}(R)$  on $\Omega_R/dR$. 
\end{prop}
\begin{proof}
First we show that the above action preserves the left $R$-submodule $K$:  For $r,a,b\in R$ and $D\in\text{Der}(R)$ we have
\begin{align*}
D(r(f\otimes g+g\otimes f-1\otimes fg) & =
D(rf\otimes g+rg\otimes f-r\otimes fg)\\
&=D(r)f\otimes g+rD(f)\otimes g+rf\otimes D(g) \\
&\quad+D(r)g\otimes f +rD(g)\otimes f+rg\otimes D(f) \\ 
&\quad -D(r)\otimes fg-r\otimes D(f)g-r\otimes fD(g) \\ \\
&=D(r)(f\otimes g+ g\otimes f- \otimes fg) \\
&\quad +r(D(f)\otimes g+g\otimes D(f) -1\otimes D(f)g) \\
&\quad +r(D(g)\otimes f +f\otimes D(g)-1\otimes fD(g)) \in K.
\end{align*}
Thus the action descends to an action on $\Omega_R$.  Now since $D(dr)=D(1\otimes r+K)=1\otimes D(r)+K=d(D(r))$, $dR$ is preserved under this action of $\text{Der}(R)$ and the claim of the lemma follows.
\end{proof}

\begin{cor} If $R=\mathbb C[t,t^{-1},u|\,u^2=t^2+4t]$ is the three point algebra., then the action of $\text{Der}(R)$ on basis elements of  $\Omega_R^1/dR$, is given by 
\begin{align*}
t^kD(\omega_0)&= (\mu_{1,k-2} + \mu_{-1,k}) \omega_1   \\ 
t^kuD(\omega_0)&= 0 \\ 
t^kD(\omega_1) &= \left(\delta_{k,-1}+4\delta_{k,-2}\right)\omega_0 -(\mu_{k+2,1}+ 2\mu_{k+1,1})\omega_{1} , \\
t^kuD(\omega_1)
&=2(\mu_{k,1}+4\mu_{k+1,1})\omega_1+ (\delta_{k,-5}+6\delta_{k,-4}+8\delta_{k,-3})\omega_0
.\end{align*}
where $D=(t+2)\partial_u+u\partial_t$.
\end{cor}
\begin{proof}
Since $D(t)=u$, $D(u)=t+2$, $D(t^{-1})=-t^{-2}u$, and 
$$
D(t^{-1}u)=-t^{-2}u^2+t^{-1}(t+2)=-(1+4t^{-1})+1+2t^{-1}=-2t^{-1}
$$
\begin{align*}
t^kD(\omega_0)&= \overline{t^kD (t^{-1}\otimes t)}= \overline{-t^{k-2}u\otimes t+t^{-1}\otimes t^k u}\\
&=\overline{t\otimes t^{k-2}u+t^{-1}\otimes t^k u}\\
&= (\mu_{1,k-2} + \mu_{-1,k}) \omega_1   \\ \\
t^kuD(\omega_0)&=  \overline{t^kuD (t^{-1}\otimes t)}= \overline{-t^{k-2}u^2\otimes t+t^{-1}\otimes t^k u^2} \\
&=  \overline{-t^{k-2}(t^2+4t)\otimes t+t^{-1}\otimes t^k(t^2+4t)} \\
&=  \overline{-t^k\otimes t-4t^{k-1}\otimes t+t^{-1}\otimes t^{k+2}+4 t^{-1}\otimes t^{k+1}} \\ 
&=0 \\ \\
t^kD(\omega_1) 
&= \overline{t^kD (t^{-1}u\otimes t)}= \overline{-t^{k-2}u^2\otimes t + t^{k+1}(t+2)u\otimes t+ t^{-1}u\otimes t^k u} , \\
&= \overline{-t^{k-2}(t^2+4t)\otimes t + t^{k+1}(t+2)u\otimes t+ t^{-1}u\otimes t^k u} \\
&= \overline{-t^{k}\otimes t -4t^{k+1}\otimes t + t^{k+1}(t+2)u\otimes t+ t^{-1}u\otimes t^k u} \\
&=-\left(\delta_{k,-1}+4\delta_{k,-2}\right)\omega_0 + t^{k+2}u\otimes t+ 2t^{k+1}u\otimes t+ t^{-1}u\otimes t^k u \\ 
&=-\left(\delta_{k,-1}+4\delta_{k,-2}\right)\omega_0 -(\mu_{k+2,1}+ 2\mu_{k+1,1})\omega_{1} \\ \\
t^kuD(\omega_1)
&= \overline{-t^{k-2}u(t^2+4t)\otimes t + t^{k+1}(t+2)u^2\otimes t+ t^{-1}u^2\otimes t^k u} \\
&= \overline{-u(t^k+4t^{k+1})\otimes t + t^{k+1}(t+2)(t^2+4t)\otimes t+ t^{-1}u^2\otimes t^k u} \\
&= \overline{-u(t^k+4t^{k+1})\otimes t + t^{k+1}(t^3+6t^2+8t)\otimes t+ (t+4)\otimes t^k u} \\
&(\mu_{k,1}+4\mu_{k+1,1})\omega_1+ (\delta_{k,-5}+6\delta_{k,-4}+8\delta_{k,-3})\omega_0+ (\mu_{1,k}+4\mu_{0,k})\omega_1 \\
\end{align*}

\end{proof}

%

 \section{ Oscillator algebras}
 \subsection{The $\beta-\gamma$ system}   The following construction in the physics literature is often called the $\beta-\gamma$ system which corresponds to our $a$ and $a^*$ below.
 Let $\hat{\mathfrak a}$ be the infinite dimensional oscillator algebra with generators $a_n,a_n^*,a^1_n,a^{1*}_n,\,n\in\mathbb Z$ together with $\mathbf 1$ satisfying the relations 
\begin{gather*}
[a_n,a_m]=[a_m,a_n^1]=[a_m,a_n^{1*}]=[a^*_n,a^*_m]=[a^*_n,a^1_m]=[a^*_n,a^{1*}_{m}]=0,\\
[a_n^{1},a_m^{1}]=[a_n^{1*},a_m^{1*}]=0=[\mathfrak a,\mathbf 1], \\
[a_n,a_m^*]=\delta_{m+n,0}\mathbf 1=[a^1_n,a_m^{1*}].
\end{gather*}
For  $c=a,a^1$ and respectively $X=x,x^1$ with $r=0$ or $r=1$, we define $\mathbb C[\mathbf x]:= \mathbb C[x_n,x_n^1\,|\,n\in\mathbb Z]$ and $\rho:\hat{\mathfrak a}\to \mathfrak{gl}(\mathbb C[\mathbf x])$ by
\begin{align}
\rho_r( c_{m}):&=\begin{cases}
  \partial/\partial
X_{m}&\quad \text{if}\quad m\geq 0,\enspace\text{and}\enspace  r=0
\\ X_{m} &\quad \text{otherwise},
\end{cases}\label{c}
 \\
\rho_r(c_{m}^*):&=
\begin{cases}X_{-m} &\enspace \text{if}\quad m\leq
0,\enspace\text{and}\enspace r=0\\ -\partial/\partial
X_{-m}&\enspace \text{otherwise}. \end{cases}\label{c*}
\end{align}
and $\rho_r(\mathbf 1)=1$.
These two representations can be constructed using induction:
For $r=0$ the representation 
$\rho_0$ is the
$\hat{\mathfrak a}$-module generated by $1=:|0\rangle$, where
$$
a_{m}|0\rangle=a^1_{m}|0\rangle=0,\quad m\geq  0,
\quad a_{m}^*|0\rangle= a_{m}^{1*}|0\rangle=0,\quad m>0.
$$
For $r=1$ the representation 
$\rho_1$ is the
$\hat{\mathfrak a}$-module generated by $1=:|0\rangle$, where
$$
\quad a_{m}^*|0\rangle= a_{m}^{1*}|0\rangle=0,\quad m\in\mathbb Z.
$$
If we define
\begin{equation}
 \alpha(z):=\sum_{n\in\mathbb Z}a_nz^{-n-1},\quad  \alpha^*(z):=\sum_{n\in\mathbb Z}a_n^*z^{-n}, \label{alpha}
\end{equation}
and
\begin{equation}
 \alpha^1(z):=\sum_{n\in\mathbb Z}a^1_nz^{-n-1},\quad  \alpha^{1*}(z):=\sum_{n\in\mathbb Z}a^{1*}_nz^{-n}, \label{alpha1}
\end{equation}
then 
\begin{align*}
[\alpha(z),\alpha(w)]&=[\alpha^*(z),\alpha^*(w)]=[\alpha^{1}(z),\alpha^{1}(w)]=[\alpha^{1*}(z),\alpha^{1*}(w)]=0  \\
[\alpha(z),\alpha^*(w)]&=[\alpha^1(z),\alpha^{1*}(w)]
    =\mathbf 1\delta(z/w).
\end{align*}
Observe that $\rho_1(\alpha(z))$ and 
$\rho_1(\alpha^1(z))$ are not fields whereas $\rho_r(\alpha^*(z))$ and $\rho_r(\alpha^{1*}(z))$
are always fields.   
Corresponding to these two representations there are two possible normal orderings:  For $r=0$ we use the usual normal ordering and for $r=1$ we define the {\it natural normal ordering} to be 
\begin{alignat*}{2}
\alpha(z)_+&=\alpha(z),\quad &\alpha(z)_-&=0 \\
\alpha^1(z)_+&=\alpha^1(z),\quad &\alpha^1(z)_-&=0 \\
\alpha^*(z)_+&=0,\quad &\alpha^*(z)_-&=\alpha^*(z), \\
\alpha^{1*}(z)_+&=0,\quad &\alpha^{1*}(z)_-&=\alpha^{1*}(z) ,
\end{alignat*}

This means in particular that for $r=0$ we get 
\begin{align}
\lfloor \alpha \alpha^*  \rfloor =\lfloor \alpha(z) ,\alpha^*(w) \rfloor
&=\sum_{m\geq 0} \delta_{m+n,0}z^{-m-1}w^{-n}
=\delta_-(z/w)
=\ 
\,\iota_{z,w}\left(\frac{1}{z-w}\right)\\
\lfloor \alpha^* \alpha \rfloor
&
=-\sum_{m\geq 1} \delta_{m+n,0}z^{-m}
w^{-n-1}
=-\delta_+(w/z)=\,\iota_{z,w}\left(\frac{1}{w-z}
\right)
\end{align}
(where $\iota_{z,w}$ denotes Taylor series expansion in the ``region'' $|z|>|w|$), 
and for $r=1$ 
\begin{align}
\lfloor \alpha ,\alpha^* \rfloor
&=[\alpha(z)_-,\alpha^*(w)]=0 \\
\lfloor \alpha^*  \alpha \rfloor
&=[\alpha^*(z)_-,\alpha(w)]=
-\sum_{\in\mathbb Z} \delta_{m+n,0}z^{-m}
w^{-n-1}
=- \delta(w/z),
\end{align}
where similar results hold for $\alpha^1(z)$.
Notice that in both cases we have
$$
[\alpha(z),\alpha^*(w)]=
\lfloor \alpha(z)\alpha^*(w)\rfloor-\lfloor\alpha^*(w) \alpha(z)\rfloor=\delta(z/w).
$$

Recall that he singular part of the {\it operator product
expansion}
$$
\lfloor
ab\rfloor=\sum_{j=0}^{N-1}\iota_{z,w}\left(\frac{1}{(z-w)^{j+1}}
\right)c^j(w)
$$
completely determines the bracket of mutually local formal
distributions $a(z)$ and $b(w)$.   One writes
$$
a(z)b(w)\sim \sum_{j=0}^{N-1}\frac{c^j(w)}{(z-w)^{j+1}}.
$$

\subsection{The $3$-point Heisenberg algebra} 
The Cartan subalgebra $\mathfrak h$ tensored with $\mathcal R$ generates a subalgebra of $\hat{{\mathfrak g}}$ which is an extension of an oscillator algebra.    This extension motivates the following definition:  The Lie algebra with generators $b_{m},b_m^1$, $m\in\mathbb Z$, $\mathbf 1_0,\mathbf 1_1 $, and relations
\begin{align*}
[b_{m},b_{n}]&=(n-m)\,\delta_{m+n,0}\mathbf 1_0=-2m\,\delta_{m+n,0}\mathbf 1_0 \\
[b^1_m,b_n^1] &=(n-m)\left(\delta_{m+n,-2}+4\delta_{m+n,-1}\right)\mathbf 1_0   \\ &= 2 \left((n+1)\delta_{m+n,-2}+(4n+2)\delta_{m+n,-1}\right)\mathbf 1_0 \notag \\
[b^1_m,b_n] &=2\mu_{m,n}\mathbf 1_1 = -[b_n, b^1_m]
  \\
[b_{m},\mathbf 1_0]&=[b_{m}^1,\mathbf 1_0]=[b_{m},\mathbf 1_1]=[b_{m}^1,\mathbf 1_1]= 0. 
\end{align*}
is the {\it $3$-point (affine) Heisenberg algebra} which we denote by $\hat{\mathfrak h}_3$.

If we introduce the formal distributions
\begin{equation} 
\beta(z):=\sum_{n\in\mathbb Z} b_nz^{-n-1},\quad \beta^1(z):=\sum_{n\in\mathbb Z}b_n^1z^{-n-1}=\sum_{n \in\mathbb Z}b_{n+\frac{1}{2}}z^{-n-1}.
\end{equation}
(where $b_{n+\frac{1}{2}}:=b^1_n$)
then the relations above can be rewritten in the form
\begin{align*}\label{bosonrelations}
[\beta(z),\beta(w)]&=2\mathbf 1_0\partial_z\delta(z/w)=-2\partial_w\delta(z/w)\mathbf 1_0 \\
[\beta^1(z),\beta^1(w)]
&=-2\left((w^2+4w) \partial_w(\delta(z/w)+ (2+w) \delta(z/w)\right)\mathbf 1_0 \\
[\beta(z),\beta^1(w)]&= -\sqrt{1+(4/w)}w \partial_w\delta(z/w)\mathbf 1_1
\end{align*}

 Set
\begin{align*}
\hat{\mathfrak h}_3^\pm:&=\sum_{n\gtrless 0}\left(\mathbb Cb_n+\mathbb Cb_n^1\right),\quad
\hat{ \mathfrak h}_3^0:=  \mathbb C\mathbf 1_0\oplus \mathbb C\mathbf 1_1\oplus \mathbb Cb_0\oplus \mathbb Cb^1_0.
\end{align*}
We introduce a Borel type subalgebra
\begin{align*}
\hat{\mathfrak b}_3&= \hat{\mathfrak h}_3^+\oplus \hat{\mathfrak h}_3^0.
\end{align*}
That $\hat{\mathfrak b}_3$ is a subalgebra follows from the above defining relations.

\begin{lem}\label{heisenbergprop}
Let $\mathcal V=\mathbb C\mathbf v_0\oplus \mathbb C\mathbf v_1$ be a two dimensional representation of 
$\hat{\mathfrak h}_3^+ $ with $\hat{\mathfrak h}_3^+\mathbf v_i=0$ for $i=0,1$.   Fix  $ B_0, B^1_{i,j}$ for $i,j = 0,1$ with $B^1_{00} = B^1_{11}$ and $ \chi_1,\kappa_0 \in \mathbb C$  and let 
\begin{align*}
b_0\mathbf v_0&=B_0 \mathbf v_0,  &b_0\mathbf v_1&=B_0 \mathbf v_1 \\
b_0^1\mathbf v_0&=B^1_{00} \mathbf v_0+B^1_{01}\mathbf v_1,  &b_0^1\mathbf v_1&=B^1_{10} \mathbf v_0+B^1_{11}\mathbf v_1\\
\mathbf 1_1\mathbf v_i&=\chi _1  \mathbf v_i,\quad   &\mathbf 1_0\mathbf v_i&=\kappa _0\mathbf v_i,\quad i=0,1.
\end{align*}
When $\chi_1$ acts as zero, the above defines a representation of  $\hat{\mathfrak b}_3$ on $\mathcal V$. 
\end{lem}

Let $ \mathbb C[\mathbf y]:= \mathbb C[y_{-n}, y_{-m}^1 | m,n \in \mathbb N^*] $. The following is a straightforward computation, with corrections to the version in \cite{MR3245847} (where some formulas for the 4-point algebra were inadvertently included).

\begin{lem} [\cite{MR3245847}] \label{rhorep}The linear map $\rho:\hat{\mathfrak b}_3\to \text{End}(\mathbb C[\mathbf y]\otimes \mathcal V)$ defined  by 
\begin{align*}
\rho(b_{n})&=y_{n} \quad \text{ for }n<0 \\
\rho(b_{n}^1)&=y_{n}^1\quad \text{ for }n<0 \\
\rho (b_n) &= -n 2 \partial_{ y_{-n} }\kappa_0   \quad \text{ for }n>0 \\
\rho(b^1_n)&= -(2+ 2n) \partial_{y^1_{-2-n}} \kappa_0 -4(1+2n) \partial_{y^1_{-1-n}} \kappa_0 \quad \text{ for }n>0\\
\rho(b^1_0)&= -2 \partial_{y^1_{-2}} \kappa_0 -4 \partial_{y^1_{-1}} \kappa_0 +B_0^1 \\
\rho(b_{0})&=B_0
\end{align*}
is a representation of $\hat{\mathfrak b}_3$.
\end{lem}

 \section{$3$-point Virasoro Algebra and its action on a Fock space $\mathcal F$}

\subsection{The 3-point Witt algebra representation}
We now construct a representation using the oscillator algebra. 
Define $\pi:\text{Der}(R)\to \text{End}(\mathbb C[\mathbf x])$ by the following manner. 
We then have in terms of formal power series \eqnref{alpha} and \eqnref{alpha1}
\begin{align} 
\pi(d)(z):
&=P(z)\left(:\alpha(z)\partial_z\alpha^*(z): +:\alpha^1(z)\partial_z\alpha^{1*}(z):\right )  \\  
& \quad +\frac{1}{2}\partial_zP(z):\alpha^1(z)\alpha^{1*}(z):  \notag \\
  \pi(d^1)(z):  
	&=:\alpha^1(z)\partial_z\alpha^*(z):+P(z):\alpha(z)\partial_z\alpha^{1*}(z):  \\
	&  \quad  +\frac{1}{2}\partial_zP(z)    
	:\alpha(z)\alpha^{1*}(z):  \notag
\end{align}

\subsection{The 3-point Virasoro algebra}
The $3$-point Virasoro algebra $\mathfrak V$ is defined to be the universal central extension of $3$-point Witt algebra $\mathfrak W$, 
\begin{equation}\label{DefVir}
\mathfrak V=\mathfrak W\oplus \mathbb Cc_1\oplus \mathbb Cc_2
\end{equation}
where we distinguish the basis elements $\mathbf d_n$ of $\mathfrak V$, from the $d_n$ of $\mathfrak W$. The relations are 
\begin{align}
[\mathfrak V,\mathbb Cc_1\oplus \mathbb Cc_2]&=0,
\end{align}
with the defining relations given in terms of generating functions
\begin{align}
[\bar{\mathbf d}^1(z),\bar{\mathbf d}^1(w)]\label{eezw}
&=\partial_w\bar{\mathbf d}(w)\delta(z/w)+2\bar{\mathbf d}(w)\partial_w\delta(z/w) \\
&\quad -\left(P(w)\partial_w^3\delta(z/w)+\dfrac{3}{2}P'(w)\partial_w^2\delta(z/w)\right) c_1,\notag \\ \notag \\
\label{ddzw}[\bar{\mathbf d}(z),\bar{\mathbf d}(w)]
&=P(w)\partial_w \bar{\mathbf d}(w)\delta(z/w)+\partial_wP(w)\bar{ \mathbf d}(w)\delta(z/w) +2P(w) \bar{\mathbf d}(w)\partial_w\delta(z/w)\\ 
&\quad -\left(P(w)^2\partial_w^3\delta(z/w)+3P'(z)P(z)\partial_w^2\delta(z/w)+6P(z)\partial_w\delta(z/w)+12\partial_w\delta(z/w)\right)c_1,\notag \\\notag \\
\label{dezw}  [ \bar{\mathbf d}(z),\bar{\mathbf d}^1(w)] &=P(w)\partial_w\bar{\mathbf d}^1(w)\delta(z/w)  +2P(w)\bar{\mathbf d}^1(w)\partial_w\delta(z/w)+\frac{3}{2}P'(w)\bar{\mathbf d}^1(w)\delta(z/w)\\
&\quad +\left(3w(2+w)(1+(4/w))^{1/2}\partial_w^2\delta(z/w)+w^3(1+(4/w))^{3/2} \partial_w^3\delta(z/w) 
\right)c_2, \notag
\end{align}
where we have used the following result: The Taylor series expansion of $\sqrt{1+z}$ in the formal power series ring $\mathbb C[\![z]\!]$ is 
\begin{equation}
1+\frac{z}{2}+\sum_{n\geq 2}(-1)^{n-1}\frac{(2n-3)!!}{2^nn!}z^n.
\end{equation}

\begin{thm}[\cite{MR3478523}]\label{mainresult2} Suppose $\lambda,\mu,\nu,\varkappa, \chi_1,\kappa_0 \in \mathbb C$ are constants with $\kappa_0\neq 0$.  The following defines a representation of the $3$-point Virasoro algebra $\mathfrak V$ on  $\mathcal F:=\mathbb C[\mathbf x]\otimes \mathbb C[\mathbf y]\otimes \mathcal V$, with $\mathcal V$ as in \lemref{heisenbergprop}
\begin{align*}
\pi(\bar {\mathbf d})(z)&=\pi(  d)(z)+\gamma :\beta(z)^2:+\mu\partial_z\beta(z)+\gamma_1:(\beta^1(z))^2:  +\gamma_2\beta(z)\\ 
\pi(\bar{\mathbf d}^1)(z)&=\pi(  d^1)(z)+\nu :\beta(z)\beta^1(z):+\zeta\partial_z\beta^1(z)  \\
\pi(  c_1)&=-\Big(\frac{1}{3}\delta_{r,0}+\frac{2}{3}\nu^2\kappa_0^2-2\zeta^2\kappa_0 \Big)=-\frac{1}{3}\left(\delta_{r,0}+8\kappa _0^4 \nu ^4\right) =-\frac{1}{3}\left(\delta_{r,0}+\frac{1}{2}\right)  \\
\pi( c_2)&=0.
\end{align*}

Where the following conditions are satisfied:
\begin{align}
\nu^2 & =\kappa_0^{-2}/4 ,   \zeta=0, \\
\gamma&=-\nu^2P(z)\kappa_0=-\frac{P(z)}{4\kappa_0},\\ \mu &=0,\quad  \gamma_1=-\nu^2\kappa_0=-\frac{1}{4\kappa_0},\quad \gamma_2=0.
\end{align}
\end{thm}

 \section{$3$-point Current Algebra and its action on a Fock space $\mathcal F$}

  Assume that $\chi_0\in\mathbb C$ and define $\mathcal V$ as in \lemref{heisenbergprop}. The $\alpha(z),\alpha^1(z),\alpha^* (z)$ and $\alpha^{1*}(z)$ are generating series of oscillator algebra elements as in \eqref{alpha} and \eqref{alpha1}. Our main result is the following 
\begin{thm} \label{mainresult1}  Fix $r\in\{0,1\}$, which then fixes the corresponding normal ordering convention defined in the previous section.  The 3-point current algebra is defined to be $\hat{{\mathfrak g}} =\left(\mathfrak{sl}(2,\mathbb C)\otimes \mathcal R\right)\oplus \mathbb C\omega_0\oplus \mathbb C\omega_1$. Then using \eqnref{c}, \eqnref{c*} and \lemref{rhorep}, the following defines a representation of the three point algebra $\hat{\mathfrak g}$ on $\mathcal F=\mathbb C[\mathbf x]\otimes \mathbb C[\mathbf y]\otimes \mathcal V$:
\begin{align*}
\tau(\omega_1)&=0, \qquad
\tau(\omega_0)=\chi_0=\kappa_0+4\delta_{r,0} ,  \\ 
\tau(f(z))&=-\alpha(z), \qquad
\tau(f^1(z))=- \alpha^1(z),   \\ \\
\tau(h(z))
&=2\left(:\alpha(z)\alpha^*(z):+:\alpha^1(z)\alpha^{1*}(z): \right)
     +\beta , \\  \\
\tau(h^1(z))
&=2\left(:\alpha^1(z)\alpha^*(z): +(z^2+4z):\alpha(z)\alpha^{1*}(z): \right) +\beta^1(z),  \\  \\
\tau(e(z)) 
&=:\alpha(z)(\alpha^*(z))^2 :+(z^2+4z):\alpha(z)(\alpha^{1*}(z))^2: +2 :\alpha^1(z)\alpha^*(z)\alpha^{1*}(z): \\
&\quad +\beta(z)\alpha^*(z)+\beta^1(z)\alpha^{1*}(z) +\chi_0\partial\alpha^* (z) \\ \\
\tau(e^1(z))  
&=\alpha^1(z)\alpha^*(z)\alpha^* (z)
 	+(z^2+4z)\left(\alpha^1(z) (\alpha^{1*} (z))^2 +2 : \alpha (z)\alpha^{*} (z)\alpha^{1*}(z):\right)  \\
&\quad +\beta^1(z) \alpha^* +(z^2+4z)\beta(z) \alpha^{1*}(z) +\chi_0\left((z^2+4z)   \partial_z\alpha^{1*}(z)   +(z+2)   \alpha^{1*}(z) \right) .
\end{align*}
\end{thm}

 \section{The $3$-point gauge algebra}
 
 We now set $\mathfrak g = \mathfrak{sl}_2(\mathbb C)$, with the usual basis $ \{e,f,h\}$ and consider the three point current algebra $\hat {\mathfrak g}$ defined by equation 
\eqref{3pointDef}. In \cite{MR3478523} an analog of the Virasoro algebra is constructed, having generators and relations given below.

Fix the following basis elements of $\text{Der}_{\mathbb C}R$:
\begin{equation}
\mathbf d_n:=t^nuD,\quad \mathbf d_n^1=t^nD,\quad D=(t+2)\frac{\partial}{\partial u}+u\frac{\partial}{\partial t}
\end{equation}

\begin{lem} The basis elements listed above for $\text{Der}_{\mathbb C}R$ and $x=e,f,h$ satisfy
\begin{align*}
[\mathbf d_m,  x_n]&=nx_{m+n+1}+4nx_{m+n}\\
[\mathbf d_m,  x'_n]&=(n+1)x'_{m+n+1}+2(2n+1)x'_{m+n}\\ \\ 
[{{\mathbf d}^1}_m,  x_n]&=nx'_{m+n-1}  \\
[{{\mathbf d}^1}_m,  x'_n]&=(n+1)x_{m+n+1}+2(2n+1)x_{m+n}   \\
\end{align*}

\end{lem}
\begin{proof}  The proof is straightforward and is left the reader.
%
%
\end{proof}
Setting  $\bar{\mathbf d}_m:=- {\mathbf d}_{m+1}$ and $\bar{\mathbf d^1}_m=-{\mathbf d^1}_{m+1}$ and  $$
\bar{\mathbf d}(z):=\sum_{m\in\mathbb Z}\bar{\mathbf d}_mz^{-m-2}
,\quad \bar{\mathbf d^1}(z):=\sum_{m\in\mathbb Z}\bar{\mathbf d^1}_mz^{-m-2}
$$
then above can be written in terms of commutators of formal sums as follows: 
\begin{align*}
[\mathbf d_m,  x(w)]
&= -w^{m}((w^2+4w)\partial_w +m(w+4) +2(w+2))x(w)\\ 
[\mathbf d_m,  x'(w)]
&=-w^{m}((w^2+4w)\partial_w +m(w+4)+w+2)x'(w) \\
[{\mathbf d^1}_m,  x(w)]
&= -w^{m-1}(w\partial_w+m)x'(w)\\
[{\mathbf d^1}_m,  x'(w)]
&=-w^{m}((w^2+4w)\partial_w +m(w+4)+w+2)x(w) \\
\end{align*}

We define the $3$-point gauge algebra to be the semi-direct product of the $3$-point Virasoro algebra with the $3$-point current algebra (see \cite{MR1629472} for the definition of the affine gauge algebra).  Besides the center acting trivially the defining relations are 
\begin{align}\label{currentalgebra1}
[\mathbf d^1(z),x'(w)]
&=P(w)\partial_wx(w)\delta(z/w)  +P(w)x(w)\partial_w\delta(z/w)   +(w+2)x(w)\delta(z/w)    \\
 \label{currentalgebra2}[\mathbf d^1(z),x(w)]
&=\partial_wx'(w)\delta(z/w) +x'(w)\partial_w\delta(z/w)  \\
\label{currentalgebra3}
[\bar{\mathbf d}(z),x'(w)]
&=P(w)\partial_wx'(w)\delta(z/w)  +P(w)x'(w)\partial_w\delta(z/w)   +(w+2)x'(w)\delta(z/w)    \\
\label{currentalgebra4}[\bar{\mathbf d}(z),x(w)]
&=P(w)\partial_wx(w)\delta(z/w)  +P(w)x(w)\partial_w\delta(z/w) +2(w+2)x(w)\delta(z/w)     .
\end{align}

\begin{thm}\label{mainresult}  The realizations in \thmref{mainresult2} and \thmref{mainresult1} give rise to the $3$-point gauge algebra acting on the fock space $\mathcal W$, under the conditions 
\begin{align}
\nu^2 & =\kappa_0^{-2}/4 ,   \zeta=0, \\
\gamma&=-\nu^2P(z)\kappa_0=-\frac{P(z)}{4\kappa_0},\\ \mu &=0,\quad  \gamma_1=-\nu^2\kappa_0=-\frac{1}{4\kappa_0},\quad \gamma_2=0.
\end{align}
and $r=1$.
\end{thm}
\begin{proof}
We need to check that the defining relations \eqnref{currentalgebra1}-\eqnref{currentalgebra4} are satisfied.  First up is 
\begin{align*}
[\pi(\bar{\mathbf d})_\lambda \tau(f)]&=\left[\left(\pi(  d)+\gamma :\beta^2: +\gamma_1:(\beta^1)^2:   \right)_\lambda (-\alpha)\right] \\
&=-\left[\left( P\left(:\alpha\partial\alpha^*: +:\alpha^1\partial \alpha^{1*}:\right )+\frac{1}{2}\partial P:\alpha^1\alpha^{1*}:\right)_\lambda \alpha\right] \\
&=-\partial P\alpha-P\partial \alpha -P\alpha\lambda \\
&=\partial P\tau(f)+P\partial \tau(f)+P\tau(f)\lambda.
\end{align*}
Next we have 
\begin{align*}
[\pi(\bar{\mathbf d})_\lambda \tau(f^1)]&=\left[\left(\pi(  d)+\gamma :\beta^2: +\gamma_1:(\beta^1)^2:   \right)_\lambda (-\alpha^1)\right] \\
&=-\left[\left( P\left(:\alpha\partial\alpha^*: +:\alpha^1\partial \alpha^{1*}:\right )+\frac{1}{2}\partial P:\alpha^1\alpha^{1*}:\right)_\lambda \alpha^1\right] \\
&=-\partial P\alpha^1-P\partial \alpha^1 -P\alpha^1\lambda+\frac{1}{2}\partial P\alpha^1  \\
&=\frac{1}{2}\partial P\tau(f^1)+P\partial \tau(f^1)+P\tau(f^1)\lambda.
\end{align*}
And
\begin{align*}
[\pi(\bar{\mathbf d}^1)_\lambda \tau(f)]&=\left[\left(\pi(  d^1)+\nu :\beta\beta^1:\right)_\lambda (-\alpha)\right] \\
&=-\left[\left(:\alpha^1\partial \alpha^*:+P:\alpha\partial \alpha^{1*}:+\frac{1}{2}\partial P
	:\alpha\alpha^{1*}:\right)_\lambda \alpha\right] \\
&=- \alpha^1\lambda-\partial \alpha^1   \\
&= \tau(f^1)\lambda+\partial\tau(f^1)  
\end{align*}
Next we have 
\begin{align*}
[\pi(\bar{\mathbf d^1})_\lambda \tau(f^1)]&=\left[\left(\pi(  d^1)+\nu :\beta\beta^1:   \right)_\lambda (-\alpha^1)\right] \\
&=-\left[\left(:\alpha^1\partial \alpha^*:+P:\alpha\partial \alpha^{1*}:+\frac{1}{2}\partial P
	:\alpha\alpha^{1*}:\right)_\lambda \alpha^1\right] \\
&=-\partial P\alpha-P\partial \alpha -P\alpha\lambda+\frac{1}{2}\partial P\alpha  \\
&=\frac{1}{2}\partial P\tau(f)+P\partial \tau(f)+P\tau(f)\lambda.
\end{align*}

Using $-4\kappa_0 \gamma=P$ we have
\begin{align*}
[\pi(\bar{\mathbf d})_\lambda \tau(h)]&=\Big[P:\alpha\partial\alpha^*:   _\lambda  2:\alpha\alpha^*:\Big]+\Big[P:\alpha^1\partial\alpha^{1*}:  _\lambda  2:\alpha^1\alpha^{1*}:\Big]+\dfrac{1}{2}\Big[\partial P:\alpha^1\alpha^{1*}:  _\lambda 2:\alpha^1\alpha^{1*}: \Big]\\
&\quad + \Big[\gamma :\beta^2: _\lambda  \beta\Big]+ \Big[\gamma_1:(\beta^1)^2: _\lambda  \beta\Big]\\
&=(2P : \partial\alpha^*  \alpha:   +2P\lambda :\alpha \alpha^*:+2\partial P: \alpha \alpha^*:+2P:\partial\alpha \alpha^*:)\\
&\quad+(2P : \partial\alpha^{1*} \alpha^1:   +2P:\alpha^1 \alpha^{1*}:\lambda+2\partial P: \alpha^1 \alpha^{1*}:+2P:\partial\alpha^1 \alpha^{1*}:)\\
&\quad+0-4\kappa_0\gamma(\beta\lambda+\partial\beta)+0\\
&=\partial P(2:\alpha \alpha^*:+2:\alpha^1 \alpha^{1*}:+\beta)\\
&\quad+P(2 : \partial\alpha^*  \alpha:+2:\partial\alpha \alpha^*:+2:\partial\alpha^{1*}  \alpha^1: +2:\partial\alpha^1 \alpha^{1*}:+\partial\beta)\\
&\quad+P\lambda(2:\alpha \alpha^*:+2:\alpha^1 \alpha^{1*}:+\beta)\\
&=\partial P\tau(h)+P\partial\tau(h)+P\tau(h)\lambda
\end{align*}
and using $-4\gamma_1 P=1$ we have
\begin{align*}
[\pi(\bar{\mathbf d})_\lambda \tau(h^1)]&=\Big[P:\alpha\partial\alpha^*:   _\lambda 2:\alpha^1\alpha^*:\Big]+\Big[P:\alpha^1\partial\alpha^{1*}:  _\lambda 2:\alpha^1\alpha^*:\Big]+\dfrac{1}{2}\Big[\partial P:\alpha^1\alpha^{1*}:  _\lambda 2:\alpha^1\alpha^*: \Big]\\
&\qquad+\Big[P:\alpha\partial\alpha^*:   _\lambda  2P:\alpha\alpha^{1*}:\Big]+\Big[P:\alpha^1\partial\alpha^{1*}:  _\lambda  2P:\alpha\alpha^{1*}:\Big]\\
&\qquad +\dfrac{1}{2}\Big[\partial P:\alpha^1\alpha^{1*}:  _\lambda 2P:\alpha\alpha^{1*}: \Big]+ \Big[\gamma :\beta^2: _\lambda  \beta^1\Big]+ \Big[\gamma_1:(\beta^1)^2: _\lambda  \beta^1\Big]\\
&=(2P:\partial\alpha^*\alpha^1:)+(2P:\alpha^1\alpha^*:\lambda+2P:\partial\alpha^1\alpha^*:+2\partial P:\alpha^1\alpha^*:)-(\partial P:\alpha^1\alpha^*:)\\
&\quad+(2P^2:\alpha\alpha^{1*}:\lambda+2P^2:\partial\alpha\alpha^{1*}:+2P\partial P:\alpha\alpha^{1*}:)+(2P^2:\partial\alpha^{1*}\alpha:)+(P\partial P:\alpha^{1*}\alpha:)\\
&\quad+0-4\gamma_1P(P\beta^1\lambda+\dfrac{1}{2}\partial P\beta^1+P\partial\beta^1)\\
&=P\lambda(2:\alpha^1\alpha^*:+2P:\alpha\alpha^{1*}:+\beta^1)+\dfrac{1}{2}\partial P(2:\alpha^1\alpha^*:+6P:\alpha\alpha^{1*}:+\beta^1)\\
&\quad+P(2:\partial\alpha^*\alpha^1:+2:\partial\alpha^1\alpha^*:+2P:\partial\alpha\alpha^{1*}:+2P:\partial\alpha^{1*}\alpha:+\partial\beta^ 1)\\
&=P\lambda(2:\alpha^1\alpha^*:+2P:\alpha\alpha^{1*}:+\beta^1)+\dfrac{1}{2}\partial P(2:\alpha^1\alpha^*:+2P:\alpha\alpha^{1*}:+\beta^1)\\
&\quad+P(2:\partial\alpha^*\alpha^1:+2:\partial\alpha^1\alpha^*:+2P:\partial\alpha\alpha^{1*}:+2P:\partial\alpha^{1*}\alpha:+2\partial P:\alpha\alpha^{1*}:+\partial\beta^1)\\
&=\dfrac{1}{2}\partial P \tau(h^1)+P\partial \tau(h^1)+P\tau(h^1)\lambda
\end{align*}

Now using $\nu\kappa_0=-\dfrac{1}{2}$
\begin{align*}
[\pi(\bar{\mathbf d^1})_\lambda \tau(h)]&=\Big[:\alpha^1\partial\alpha^*:   _\lambda  2:\alpha\alpha^*:\Big]+\Big[:\alpha^1\partial\alpha^*:   _\lambda  2:\alpha^1\alpha^{1*}:\Big]\\
&\qquad+\Big[P:\alpha\partial\alpha^{1*}:  _\lambda  2:\alpha\alpha^*:\Big]+\Big[P:\alpha\partial\alpha^{1*}:  _\lambda  2:\alpha^1\alpha^{1*}:\Big]\\
&\qquad+\dfrac{1}{2}\Big[\partial P:\alpha\alpha^{1*}:  _\lambda 2:\alpha\alpha^{*}: \Big]+\dfrac{1}{2}\Big[\partial P:\alpha\alpha^{1*}:  _\lambda 2:\alpha^1\alpha^{1*}: \Big]\\
&\qquad + \Big[\nu :\beta\beta^1:_\lambda \beta\Big] \\
&=(2\lambda:\alpha^1\alpha^*:+2:\partial\alpha^1\alpha^*:)+(2:\partial\alpha^{*}\alpha^ 1:)\\
&\quad +(2P:\partial\alpha^{1*}\alpha:)+(2P\lambda:\alpha\alpha^{1*}:+2P:\partial\alpha\alpha^{1*}:+2\partial P:\alpha\alpha^{1*}:)\\
&\quad+(\partial P:\alpha^{1*}\alpha:)-(\partial P:\alpha\alpha^{1*}:) -2\nu\kappa_0(\beta^1\lambda+\partial\beta^1)\\
&=(2:\alpha^1\alpha^*:+2P:\alpha\alpha^{1*}:-2\nu\kappa_0\beta^1)\lambda\\
&\quad+(2:\partial\alpha^1\alpha^*:+2:\partial\alpha^ *\alpha^{1}:+2\partial P:\alpha\alpha^{1*}:+2P:\partial\alpha^{1*}\alpha:+2P:\partial\alpha\alpha^{1*}:-2\nu\kappa_0\partial\beta^1)\\
&=\tau(h^1)\lambda+\partial \tau(h^1)
\end{align*}

Next on the agenda is
\begin{align*}
[\pi(\bar{\mathbf d^1})_\lambda \tau(h^1)]&=\Big[:\alpha^1\partial\alpha^*:   _\lambda 2:\alpha^1\alpha^*:\Big]+\Big[P:\alpha\partial\alpha^{1*}:  _\lambda 2:\alpha^1\alpha^*:\Big]+\dfrac{1}{2}\Big[\partial P:\alpha\alpha^{1*}:  _\lambda 2:\alpha^1\alpha^*: \Big]\\
&\qquad+\Big[:\alpha^1\partial\alpha^*:   _\lambda  2P:\alpha\alpha^{1*}:\Big]+\Big[P:\alpha\partial\alpha^{1*}:  _\lambda  2P:\alpha\alpha^{1*}:\Big]\\
&\qquad +\dfrac{1}{2}\Big[\partial P:\alpha\alpha^{1*}:  _\lambda 2P:\alpha\alpha^{1*}: \Big]+ \Big[\nu :\beta\beta^1:_\lambda \beta^1\Big]\\
&=0+(2P:\partial\alpha^{1*}\alpha^1:+2P:\alpha\alpha^*:\lambda+2P:\partial\alpha\alpha^*:+2\partial P:\alpha\alpha^*:)\\
&\quad+(\partial P:\alpha^{1*}\alpha^1:-\partial P:\alpha\alpha^*:)\\
&\quad+(2P:\partial\alpha^*\alpha:+2P:\alpha^1\alpha^{1*}:\lambda+2P:\partial\alpha^1\alpha^{1*}:)\\
&\quad+0+0-2\nu\kappa_0( P\beta\lambda+ P\partial\beta+\dfrac{1}{2}\partial P\beta)\\
&=(P\lambda+\dfrac{1}{2}\partial P)(2:\alpha\alpha^*:+2:\alpha^1\alpha^{1*}:+\beta)\\
&\quad+P(2:\partial\alpha\alpha^*:+2:\alpha\partial\alpha^*:+2:\partial\alpha^1\alpha^{1*}:+2:\alpha^1\partial\alpha^{1*}:+\partial\beta)\\
&=\dfrac{1}{2}\partial P\tau(h)+P\partial\tau(h)+P\tau(h)\lambda\\
\end{align*}


One of the most difficult calculations are
\begin{align*}
[\pi(\bar{\mathbf d})_\lambda \tau(e)]&=\Big[\left(P\left(:\alpha\partial\alpha^*: +:\alpha^1\partial\alpha^{1*}:\right )
+\frac{1}{2}\partial P:\alpha^1\alpha^{1*}:\right) _\lambda    \\
& :\alpha(w)(\alpha^*(w))^2 :+P(w):\alpha(w)(\alpha^{1*}(w))^2: +2 :\alpha^1(w)\alpha^*(w)\alpha^{1*}(w): \\
&\quad +\beta(w)\alpha^*(w)+\beta^1(w)\alpha^{1*}(w) +\chi_0\partial\alpha^* (w)    \Big]\\
&\qquad + \Big[\left(-\nu^2P(z)\kappa_0:\beta^2:-\nu^2\kappa_0:(\beta^1)^2:\right)_\lambda  
:\alpha(w)(\alpha^*(w))^2 :+P(w):\alpha(w)(\alpha^{1*}(w))^2:  \\
& \quad +2 :\alpha^1(w)\alpha^*(w)\alpha^{1*}(w): 
 +\beta(w)\alpha^*(w)+\beta^1(w)\alpha^{1*}(w) +\chi_0\partial\alpha^* (w)     \Big] \\
\end{align*}

%
First we obtain (using wick's formula and Taylor's formula with previous relations)

\begin{align*}
  [P(z):\alpha (z ) & \partial\alpha^*( z) :  ,
 :\alpha(w)(\alpha^*(w))^2 :  ]   \\
  \\ & =  2 P(z) [ \alpha( z), \alpha^*(w) ] : \partial\alpha^*( z)  \alpha(w) \alpha^*(w): 
 + P(z)  [\partial\alpha^*( z) , \alpha(w)  ] :   \alpha ( z )\alpha^*(w)  \alpha^*(w): \\
& = 2 P(z) \delta(z/w) : \partial_z ( \alpha^*( w) + \partial_w \alpha^* (w)(z-w) + \cdots)  \alpha(w) \alpha^*(w):
 + P(z)  \partial \delta(z/w) :   \alpha ( z )\alpha^*(w) \alpha^*(w): \\
 & = 2 P(w) \delta(z/w) : \partial_w  \alpha^*( w)   \alpha(w) \alpha^*(w):
 + P(w)  \partial \delta(z/w) :   \alpha ( w )\alpha^*(w) \alpha^*(w): \\ 
 &   +  P(w) \delta(z/w) :     \partial \alpha ( w )\alpha^*(w) \alpha^*(w):+\partial P(w)   \delta(z/w) :   \alpha ( w )\alpha^*(w) \alpha^*(w):  \end{align*}

 \begin{align*}
[   P(z):\alpha( z) \partial\alpha^*( z) : , &
P(w)  :\alpha(w)(\alpha^{1*}(w))^2 :   ]    = P^2 \partial\delta(z/w) :  \alpha(w)(\alpha^{1*}(w))^2: +   P^2\delta(z/w) :   \partial\alpha(w)(\alpha^{1*}(w))^2:  \\
& +  P  \partial P \delta(z/w) :  \alpha(w)(\alpha^{1*}(w))^2:
\end{align*}

 \begin{align*}
[   P(z):\alpha( z) \partial\alpha^*( z) : , &
2  :\alpha^1(w) \alpha^* \alpha^{1*}(w) :   ]   
  = 2P \delta(z/w) :  \partial \alpha^*(w) \alpha^1 (w) \alpha^{1*}(w):
\end{align*}

 \begin{align*}
[   P(z):\alpha( z) \partial\alpha^*( z) : , 
\beta(w) \alpha^*(w)     ]  
  =  P \delta(z/w) \partial  \alpha^*(w)  \beta(w)  \end{align*}

 \begin{align*}
[   P(z):\alpha( z) \partial\alpha^*( z) : , 
\chi_0 \partial  \alpha^*( w) ]  
  =   P \partial \delta(z/w)  \partial \alpha^*( w ) \chi_0 +      P  \delta(z/w)  \partial^2 \alpha^*( w ) \chi_0  + \partial P \delta(z/w)  \partial \alpha^*( w ) \chi_0  
\end{align*}

  \begin{align*}
  [P(z) :\alpha^1 (z ) & \partial\alpha^{1*}( z) :  ,
P(w)  :\alpha(w)(\alpha^{1*}(w))^2 :  ]   \\
 & =  2 P(z)P(w) [ \alpha^1( z), \alpha^{1*}(w) ] : \partial\alpha^{1*}( z)  \alpha(w) \alpha^{1*} (w): \\
& =  2 P^2(w) \delta(z/w) : \partial \alpha^{1*}(w)  \alpha(w) \alpha^{1*} (w): 
 \end{align*}
  
    \begin{align*}
  [P(z) :\alpha^1 (z ) & \partial\alpha^{1*}( z) :  ,
 2  :\alpha^1(w) \alpha^* \alpha^{1*}(w) :  ]   \\
 & =  2P \delta(z/w) :\alpha^1(w) \alpha^*(w) \partial \alpha^{1*}(w)  : + 2P \partial \delta(z/w) : \alpha^1(w) \alpha^* \alpha^{1*}(w) :  \\
 &+2P \delta(z/w) : \partial \alpha^1(w) \alpha^* \alpha^{1*}(w) : +  2 \partial P \delta(z/w) : \alpha^1(w) \alpha^* \alpha^{1*}(w) : 
  \end{align*}

   \begin{align*}
  [P(z) :\alpha^1 (z ) & \partial\alpha^{1*}( z) :  ,
  \beta^1 (w) \alpha^{1*}(w) ]    =  P \delta(z/w) \beta^1(w) \partial \alpha^{1*}(w)   
 \end{align*}

\begin{align*}
  [ 1/2\  \partial P : \alpha^1(z)     \alpha^{1*}(z) :,   P(w)  :\alpha(w)(\alpha^{1*}(w))^2 :     ]   \\
=   \partial P \, P \delta(z/w) :\alpha(w)(\alpha^{1*}(w))^2 :  
 \end{align*}

\begin{align*}
  [ 1/2\  \partial P : \alpha^1(z)     \alpha^{1*}(z) :,   2  :\alpha^1(w) \alpha^*(w) \alpha^{1*}(w) :    ]  = 0 \\
 \end{align*}

\begin{align*}
  [ 1/2\  \partial P : \alpha^1(z)     \alpha^{1*}(z) :,    \beta^1(w) \alpha^{1*}(w)    ]   = 1/2\  \partial P \delta(z/w)    \beta^1(w)  \alpha^{1*}(w)  
 \end{align*}

\begin{align*}
  [ \gamma : \beta^2: + & \gamma_1 : (\beta^1 )^2 ,    \beta(w)  \alpha^{*}(w)   +    \beta^1(w)  \alpha^{1*}(w)     ]   \\
  & =   -4 \gamma\partial \delta(z/w) \beta \alpha^*  -4 \gamma \delta(z/w) \partial\beta \alpha^* \\
  &
  - 4 \gamma_1 ( P \partial \delta(z/w) + 1/2\, P' \delta(z/w)  )\beta^1 \alpha^{1*}  1_0 - 4 \gamma_1 P \delta(z/w)    \partial  \beta^1 \alpha^{1*} 1_0
 \end{align*}
Since $1_1 \rightarrow \chi_1$ acts as zero, and if $r = 1$, $1_0 \rightarrow \chi_0 \rightarrow \kappa_0$.

Collecting terms, and using $\gamma = -\nu^2 P \kappa_0$, $\gamma_1 = - \nu^2 \kappa_0$ gives:

\begin{align*}
[\pi(\bar{\mathbf d})_\lambda \tau(e)]&=\Big[\left(P\left(:\alpha\partial\alpha^*: +:\alpha^1\partial\alpha^{1*}:\right )
+\frac{1}{2}\partial P:\alpha^1\alpha^{1*}: + \gamma :\beta^2: +\gamma_1:(\beta^1)^2: \right) _\lambda    \\
& :\alpha(w)(\alpha^*(w))^2 :+P(w):\alpha(w)(\alpha^{1*}(w))^2: +2 :\alpha^1(w)\alpha^*(w)\alpha^{1*}(w): \\
&\quad +\beta(w)\alpha^*(w)+\beta^1(w)\alpha^{1*}(w) +\chi_0\partial\alpha^* (w)    \Big]\\
 & =   P(w) \partial \tau e(w) \delta(z/w)  P(w) \partial \delta(z/w) e(w) + \partial P  \tau(e(w  ) \delta(z/w)
 \end{align*}

Now we compute:

\begin{align*}
[\pi(\bar {\mathbf d}^1)_\lambda \tau(e)] 
& = [ :\alpha^1\partial \alpha^*:+P:\alpha\partial \alpha^{1*}:+\frac{1}{2}\partial P
	:\alpha\alpha^{1*}: +\nu :\beta\beta^1:_\lambda \\
	& :\alpha(w)(\alpha^*(w))^2 :+P(w):\alpha(w)(\alpha^{1*}(w))^2: +2 :\alpha^1(w)\alpha^*(w)\alpha^{1*}(w): \\
&\quad +\beta(w)\alpha^*(w)+\beta^1(w)\alpha^{1*}(w) +\chi_0\partial\alpha^* (w) ]
\end{align*}

Breaking the calculation into parts, keeping only the nonzero contributions:

\begin{align*}
 [ :\alpha^1(z) & \partial \alpha^*(z):,  
	  :\alpha(w)(\alpha^*(w))^2 :+P(w):\alpha(w)(\alpha^{1*}(w))^2: +2 :\alpha^1(w)\alpha^*(w)\alpha^{1*}(w): 
         + \beta^1(w)\alpha^{1*}(w)  ]\\
            =& [  \partial \alpha^*(z),\alpha(w)]  :\alpha^1(z)(\alpha^*(w))^2 : + 2 P(w)[\alpha^1(z) ,\alpha^{1*}(w)  ]
         : \partial \alpha^*(z)\alpha(w)\alpha^{1*}(w):   \\
         &+P(w) [ \partial \alpha^*(z) ,  \alpha(w)  ]:\alpha^1(z)(\alpha^{1*}(w))^2: +\color{red} \delta_{r,0}()\\
         & + 2 [\alpha^1(z) ,  \alpha^{1*}(w)]:\partial \alpha^*(z) \alpha^1(w)\alpha^*(w): + [\alpha^1(z) ,  \alpha^{1*}(w)]\partial \alpha^*(z) \beta^1(w)\\
          =  & \partial \delta(z/w)  :\alpha^1(z)(\alpha^*(w))^2 : + 2 P(w)\delta(z/w)
         : \partial \alpha^*(z)\alpha(w)\alpha^{1*}(w):   \\
         &+P(w)  \partial_z \delta(z/w):\alpha^1(z)(\alpha^{1*}(w))^2:  \\ 
         & + 2 \delta(z/w):\partial \alpha^*(z) \alpha^1(w)\alpha^*(w): +\delta(z/w) \partial \alpha^*(z) \beta^1(w)\\
           = & \partial \delta(z/w)  :\alpha^1(w)(\alpha^*(w))^2 : + \delta(z/w)  :  \partial\alpha^1(w)(\alpha^*(w))^2 :
     + 2 P(w)\delta(z/w)
         : \partial \alpha^*(w)\alpha(w)\alpha^{1*}(w):   \\
         &+P(w)  \partial \delta(z/w):\alpha^1(w)(\alpha^{1*}(w))^2:  + P(w)   \delta(z/w):\partial\alpha^1(w)(\alpha^{1*}(w))^2:\\ 
         & + 2 \delta(z/w):\partial \alpha^*(w) \alpha^1(w)\alpha^*(w): +\delta(z/w) \partial \alpha^*(w) \beta^1(w)\\
\end{align*}

\begin{align*}
 [  P(z):\alpha(z) \partial \alpha^{1*}(z): ,	& :\alpha(w)(\alpha^*(w))^2 : +2 :\alpha^1(w)\alpha^*(w)\alpha^{1*}(w): +\beta(w)\alpha^*(w)  +\chi_0\partial\alpha^* (w) ]\\
   = & 2P(z)[\alpha(z), \alpha^*(w)]  : \partial \alpha^{1*}(z)\alpha(w)\alpha^*(w) : 
 +    2P(z)[\alpha(z),  \alpha^*(w)] :\partial \alpha^{1*}(z)\alpha^1(w)\alpha^{1*}(w):  \\
 &+  2P(z)[ \partial \alpha^{1*}(z), \alpha^1(w)] :\alpha(z)\alpha^*(w)\alpha^{1*}(w): \\
 & +  P(z)[\alpha(z),\alpha^*(w)]    \partial \alpha^{1*}(z): \beta(w)  +\chi_0 P(z)[\alpha(z),\partial\alpha^* (w)]  \partial \alpha^{1*}(z) \\
  =  & 2P(z)\delta : \partial \alpha^{1*}(z)\alpha(w)\alpha^*(w) : 
 +    2P(z) \delta :\partial \alpha^{1*}(z)\alpha^1(w)\alpha^{1*}(w):  \\
 &+  2P(z) \partial\delta :\alpha(z)\alpha^*(w)\alpha^{1*}(w):  \\
 & +  P(z)\delta  \partial \alpha^{1*}(z): \beta(w)  +\chi_0 P(z) \partial\delta \partial \alpha^{1*}(z) \\
   =  & 2P(w)\delta :   \partial \alpha^{1*}(w)  \alpha(w)\alpha^*(w) : \\
 &+    2P(w) \delta :  \partial\alpha^{1*}(w)    \alpha^1(w)\alpha^{1*}(w):  \\
 & +  2\big( \partial P(w)  \delta :\alpha(w)\alpha^*(w)\alpha^{1*}(w):  + P(w)  \partial  \delta :\alpha(w)\alpha^*(w)\alpha^{1*}(w):  \\
& +    P(w)   \delta : \partial \alpha(w)\alpha^*(w)\alpha^{1*}(w):   \big) 
  +  P(w)\delta  \partial \alpha^{1*}(w): \beta(w) \\
&  +\chi_0\big( \partial P(w) \delta \partial \alpha^{1*}(w) +  P(w) \partial\delta \partial \alpha^{1*}(w)  +  P(w)  \delta \partial^2 \alpha^{1*}(w)  \big)\\
 \\
\end{align*}

\begin{align*}
 [  \frac{1}{2}\partial P 
	:\alpha\alpha^{1*}:, &:\alpha(w)(\alpha^*(w))^2 : +2 :\alpha^1(w)\alpha^*(w)\alpha^{1*}(w): +\beta(w)\alpha^*(w)  +\chi_0\partial\alpha^* (w) ]\\
	   =  & \partial P(z)[\alpha(z), \alpha^*(w)]  :  \alpha^{1*}(z)\alpha(w)\alpha^*(w) : 
 +    \partial P(z)[\alpha(z),  \alpha^*(w)] :  \alpha^{1*}(z)\alpha^1(w)\alpha^{1*}(w):  \\
 &+  \partial P(z)[  \alpha^{1*}(z), \alpha^1(w)] :\alpha(z)\alpha^*(w)\alpha^{1*}(w): \\
 & +  1/2 \, \partial P(z)[\alpha(z),\alpha^*(w)]      \alpha^{1*}(z)  \beta(w)  +1/2 \,  \chi_0 \partial P(z)[\alpha(z), \partial \alpha^* (w)] \alpha^{1*}(z)\\
   = &  \partial P(w)\delta(z/w) :  \alpha^{1*}(w)\alpha(w)\alpha^*(w) : +    \partial P(w) \delta(z/w):  \alpha^{1*}(w)\alpha^1(w)\alpha^{1*}(w): \\
 &-    \partial P(w) \delta(z/w) :\alpha(w)\alpha^*(w)\alpha^{1*}(w):  + 1/2 \, \partial P(w)\delta(z/w)    \alpha^{1*}(w)  \beta(w)  \\
 &+1/2 \,  \chi_0 (\partial^2 P(w)  \delta(z/w) \alpha^{1*}(w )  +  \partial P(w)  \partial \delta(z/w) \alpha^{1*}(w ) 
 + \partial P(w)  \delta(z/w) \partial  \alpha^{1*}(w )  ) \\
\end{align*}

\begin{align*} 
 [  \nu :\beta\beta^1:,  
	& \beta(w)\alpha^*(w)+\beta^1(w)\alpha^{1*}(w)  ]\\
  =  & \nu [\beta, \beta(w)] \beta^1\alpha^*(w) +  \nu [\beta^1,   \beta(w) ]  \beta \alpha^*(w) \\
& +  \nu [\beta, \beta^1(w)] \beta^1\alpha^{1*}(w)  +  \nu [\beta^1,   \beta^1(w) ]  \beta \alpha^{1*}(w)  \\
  =  & -2\nu \partial \delta(z/w)  \beta^1(z) \alpha^{*}(w)1_0  -2 \nu  P(w) \partial \delta  \beta(z) \alpha^{1*}(w) 1_0
 - \nu \partial P(w) \delta(z/w) \beta(z) \alpha^{1*}(w)  1_0 +\color{red} () 1_1\\
   = & -2\nu \partial \delta(z/w)  \beta^1(w) \alpha^{*}(w)1_0   -2\nu\delta(z/w)   \partial \beta^1(w) \alpha^{1*}(w)1_0 \\
& -2 \nu (  P(w) \partial \delta  \beta(w) \alpha^{1*}(w) 1_0  
+ P(w)  \delta \partial \beta(w) \alpha^{1*}(w) 1_0) \\
& - \nu \partial P(w) \delta(z/w) \beta(w) \alpha^{1*}(w)  1_0 
\end{align*}

Collecting terms, and choosing $\nu \kappa_0 = - 1/2$ (which is consistent with our previous choices) we get:

\begin{align*}
[\pi(\bar {\mathbf d}^1)_\lambda \tau(e)] 
& =  \partial \delta(z/w) \tau (e^1) +  \partial \tau e^1 \delta(z/w)
\end{align*}

\begin{align*}
[\pi(\bar{\mathbf d})_\lambda \tau(  e^1)]
&=\Big[\left(P\left(:\alpha\partial\alpha^*: +:\alpha^1\partial\alpha^{1*}:\right )
+\frac{1}{2}\partial P:\alpha^1\alpha^{1*}:\right) _\lambda    \\
&\qquad\Big(\alpha^1(\alpha^*)^2
 	+P\left(\alpha^1 (\alpha^{1*} )^2 +2 : \alpha \alpha^{*} \alpha^{1*}:\right)  \\
&\qquad +\beta^1 \alpha^* +P\beta \alpha^{1*} +\chi_0\left(P  \partial \alpha^{1*}   +\frac{1}{2}\partial P  \alpha^{1*} \right)\Big)  \Big]\\
&\qquad - \Big[\left(\nu^2P\kappa_0:\beta^2:+\nu^2\kappa_0:(\beta^1)^2:\right)_\lambda\left(  \beta^1 \alpha^* +P\beta \alpha^{1*}\right) \Big] \\ \\ 
\end{align*}

%

\begin{align*}
  P(z):\alpha(z) &\partial\alpha^* (z) :  
 :\alpha^1(w)(\alpha^*(w))^2 :     \\
 &  \sim   2 P(z) \lfloor \alpha (z) \alpha^*(w) \rfloor : \partial\alpha^*( z)  \alpha^1(w) \alpha^*(w):\\
&  \sim   2 P(z) \iota_{z,w}\left(\frac{1}{z-w}\right): \partial\alpha^*( z)  \alpha^1(w) \alpha^*(w):\\
&  \sim   2 (P(w)+(z-w)\partial_wP(w)) \iota_{z,w}\left(\frac{1}{z-w}\right): (\partial\alpha^*( w) +(z-w)\partial_w\alpha^*(w)) \alpha^1(w) \alpha^*(w):\\
&  \sim   2 P(w) \iota_{z,w}\left(\frac{1}{z-w}\right):  \alpha^1(w) \alpha^*(w)\partial\alpha^*( w) :\\
\end{align*}
Hence
\begin{align*}
[P:\alpha\partial\alpha^*:{_\lambda}:\alpha^1(\alpha^*)^2]= 2 P:  \alpha^1\alpha^*\partial\alpha^*: 
\end{align*}
Next we calculate 
\begin{align*}
  P(z):\alpha^1(z)& \partial \alpha^{1*} (z) :  
 :\alpha^1(w)(\alpha^*(w))^2 :     \\
 &  \sim P(z)\lfloor \partial \alpha^{1*} (z):   \alpha^1(w)\rfloor:\alpha^1(z) (\alpha^*(w))^2 : \\
 &  \sim (P(w)+(z-w)\partial_wP(w))\iota_{z,w}\left(\frac{1}{(z-w)^2}\right):\left(\alpha^1(w)+(z-w)\partial_w\alpha^1(w)\right) (\alpha^*(w))^2 : \\
 &  \sim (P(w)+(z-w)\partial_wP(w))\iota_{z,w}\left(\frac{1}{(z-w)^2}\right):\left(\alpha^1(w)+(z-w)\partial_w\alpha^1(w)\right) (\alpha^*(w))^2 : \\
&  \sim P(w) \iota_{z,w}\left(\frac{1}{(z-w)^2}\right):\alpha^1(w) (\alpha^*(w))^2 : \\
&  \quad+\partial_wP(w)\iota_{z,w}\left(\frac{1}{(z-w)}\right):\alpha^1(w) (\alpha^*(w))^2 : \\
&  \quad+ P(w) \iota_{z,w}\left(\frac{1}{(z-w)}\right):\partial_w\alpha^1(w)  (\alpha^*(w))^2 : \\
\end{align*}
Hence
\begin{align*}
[P:\alpha^1\partial\alpha^{1*}:{_\lambda}:\alpha^1(\alpha^*)^2]= P :\alpha^1 (\alpha^*)^2 :\lambda+\partial P:\alpha^1 (\alpha^*)^2 : + P:\partial \alpha^1  (\alpha^*)^2 : \\
\end{align*}

Next we calculate 
\begin{align*}
  \frac{1}{2}\partial_zP(z):\alpha^1(z)&  \alpha^{1*} (z) :  
 :\alpha^1(w)(\alpha^*(w))^2 :     \\
 &  \sim  - \frac{1}{2}\partial_wP(w):\alpha^1(w) \alpha^*(w))^2 :  
  \end{align*}
Hence
\begin{align*}
[ \frac{1}{2}\partial P:\alpha^1&  \alpha^{1*} :{_\lambda}  
 :\alpha^1(\alpha^*)^2 ]=  - \frac{1}{2}\partial P:\alpha^1 (\alpha^*)^2 :   \\
\end{align*}

Next we have 
\begin{align*}
  [ P:\alpha &\partial\alpha^*{_\lambda}P\alpha^1(\alpha^{1*})^2]=0.
  \end{align*}

Next we have 
\begin{align*} 
 P(z):\alpha^1(z)& \partial \alpha^{1*} (z) :  P(w):\alpha^1(w)(\alpha^{1*}(w))^2:  \\
 &\sim  P(z) P(w) \lfloor \partial \alpha^{1*} (z)\alpha^1(w)\rfloor  :\alpha^1(z)(\alpha^{1*}(w))^2 : \\ 
 &\quad +2P(z) P(w) \lfloor   \alpha^{1} (z)\alpha^{1*}(w)\rfloor  : \partial \alpha^{1*} (z)\alpha^1(w)(\alpha^{1*}(w)) : \\
 &\quad  +2P(z)P(w)\lfloor \alpha^1(z)\alpha^{1*}(w))\rfloor\lfloor \partial \alpha^{1*} (z) \alpha^1(w)\rfloor  \alpha^{1*}(w) \\  \\
 &\sim  P(z) P(w)\iota_{z,w}\left(\frac{1}{(z-w)^2}\right) :\alpha^1(z)(\alpha^{1*}(w))^2 : \\ 
 &\quad +2P(z) P(w)\left(\frac{1}{z-w}\right) : \partial \alpha^{1*} (z)\alpha^1(w)(\alpha^{1*}(w)) : \\
 &\quad  +2P(z)P(w)\iota_{z,w}\left(\frac{1}{(z-w)^3}\right) \alpha^{1*}(w) \\ \\
&\sim   P(w)^2\iota_{z,w}\left(\frac{1}{(z-w)^2}\right) :\alpha^1(w)(\alpha^{1*}(w))^2 : \\ 
&\quad +  \partial_wP(w) P(w)\iota_{z,w}\left(\frac{1}{(z-w)}\right) :\alpha^1(w)(\alpha^{1*}(w))^2 : \\ 
&\quad +  P(w)^2\iota_{z,w}\left(\frac{1}{(z-w)}\right) :\partial_w\alpha^1(w)(\alpha^{1*}(w))^2 : \\ 
 &\quad +2P(w)^2\left(\frac{1}{z-w}\right) : \alpha^1(w)\alpha^{1*}(w) \partial \alpha^{1*} (w): \\ 
 &\quad  +2P(w)^2\iota_{z,w}\left(\frac{1}{(z-w)^3}\right) \alpha^{1*}(w)+2P(w)\partial_wP(w)\iota_{z,w}\left(\frac{1}{(z-w)^2}\right) \alpha^{1*}(w) \\
&\quad+P(w)\partial^2P(w)\iota_{z,w}\left(\frac{1}{(z-w)}\right) \alpha^{1*}(w) \ \\ \\
\end{align*}

Hence
\begin{align*}
[P:\alpha^1 \partial \alpha^{1*} :&{_\lambda}  P:\alpha^1(\alpha^{1*})^2:  ]\\
&=P^2  :\alpha^1(\alpha^{1*})^2 :\lambda   +  P\partial P  :\alpha^1(\alpha^{1*})^2 : \\ 
&\quad +  P^2 :\partial \alpha^1(\alpha^{1*})^2 :   +2P^2  \alpha^1\alpha^{1*} \partial \alpha^{1*}: \\
&\quad  +2\delta_{r,0}P^2  \alpha^{1*}\lambda^2+2\delta_{r,0}P\partial P \alpha^{1*}\lambda +\delta_{r,0}P\partial^2P  \alpha^{1*}
\end{align*}

Next we have 
\begin{align*} 
 \frac{1}{2}\partial_zP(z):\alpha^1(z)&  \alpha^{1*} (z) :  P(w):\alpha^1(w)(\alpha^{1*}(w))^2:  \\
 &\sim  \frac{1}{2}\partial_z P(z)P(w) \lfloor   \alpha^{1*} (z)\alpha^1(w)\rfloor  :\alpha^1(z)(\alpha^{1*}(w))^2 : \\ 
 &\quad + \partial_z P(z)P(z) P(w) \lfloor   \alpha^{1} (z)\alpha^{1*}(w)\rfloor  :   \alpha^{1*} (z)\alpha^1(w)(\alpha^{1*}(w)) :  \\  
  &\quad+  \partial_zP(z)\lfloor   \alpha^{1*} (z)\alpha^1(w)\rfloor \lfloor   \alpha^{1} (z)\alpha^{1*}(w)\rfloor   \alpha^{1*}(w) \\ \\
 &\sim - \frac{1}{2}P(w)\partial_w P(w)\iota_{z,w}\left(\frac{1}{z-w}\right) :\alpha^1(w)(\alpha^{1*}(w))^2 : \\ 
 &\quad +P(w) \partial_wP(w)\iota_{z,w}\left(\frac{1}{z-w}\right): \alpha^1(w)(\alpha^{1*}(w))^2 :  \\  
  &\quad-  \delta_{r,0}\partial_wP(w)P(w)\iota_{z,w}\left(\frac{1}{(z-w)^2}\right) \alpha^{1*}(w) \\ 
  &\quad-  \delta_{r,0}\partial^2_wP(w)P(w)\iota_{z,w}\left(\frac{1}{(z-w)}\right) \alpha^{1*}(w) .
\end{align*}
Hence
\begin{align*} 
\frac{1}{2}[ \partial P:\alpha^1& \alpha^{1*}  :{_\lambda}   P:\alpha^1(\alpha^{1*})^2: ] = \frac{1}{2}P\partial P :\alpha^1(\alpha^{1*})^2 :  - \delta_{r,0} P\partial P \alpha^{1*}\lambda - \delta_{r,0} P\partial ^2P \alpha^{1*} .
 \end{align*}

Next we have 
\begin{align*}
 2  P(z):\alpha(z)&  \partial_z\alpha^{*} (z) :  P(w) :\alpha (w)\alpha^*(w)\alpha^{1*}(w) :     \\
&\sim 2  P(z)P(w)  \iota_{z,w}\left(\frac{1}{z-w}\right) :\alpha (w) \partial_z\alpha^{*} (z) \alpha^{1*}(w) :     \\
&\quad + 2  P(z)P(w)  \iota_{z,w}\left(\frac{1}{(z-w)^2}\right) :\alpha (z) \alpha^{*} (w) \alpha^{1*}(w) :     \\
&\quad + 2 \delta_{r,0} P(z)P(w)  \iota_{z,w}\left(\frac{1}{(z-w)^3}\right) \alpha^{1*}(w)      \\  \\
&\sim 2  P(w)^2  \iota_{z,w}\left(\frac{1}{z-w}\right) :\alpha (w) \partial_w\alpha^{*} (w) \alpha^{1*}(w) :     \\
&\quad + 2  P(w)^2  \iota_{z,w}\left(\frac{1}{(z-w)^2}\right) :\alpha (w) \alpha^{*} (w) \alpha^{1*}(w) :     \\
&\quad + 2 \partial_w P(w)P(w)  \iota_{z,w}\left(\frac{1}{(z-w)}\right) :\alpha (w) \alpha^{*} (w) \alpha^{1*}(w) :     \\
&\quad + 2  P(w)^2  \iota_{z,w}\left(\frac{1}{(z-w)}\right) :\partial_w\alpha (w) \alpha^{*} (w) \alpha^{1*}(w) :     \\
&\quad + 2  P(w)^2  \iota_{z,w}\left(\frac{1}{(z-w)^3}\right) \alpha^{1*}(w)      \\
&\quad + 2 P(w)  \partial_wP(w) \iota_{z,w}\left(\frac{1}{(z-w)^2}\right) \alpha^{1*}(w)      \\
&\quad +    P(w) \partial_w^2P(w) \iota_{z,w}\left(\frac{1}{(z-w)}\right) \alpha^{1*}(w)      \\
\end{align*}
Hence
\begin{align*}
 2  [P:\alpha&  \partial \alpha^{*}  : {_\lambda} P :\alpha \alpha^*\alpha^{1*} :  ]   \\
&= 2  P^2   :\alpha  \partial\alpha^{*}  \alpha^{1*} :   + 2  P^2  :\alpha  \alpha^{*}  \alpha^{1*} : \lambda    \\
&\quad + 2P \partial P  :\alpha  \alpha^{*}  \alpha^{1*} :     + 2  P^2   :\partial \alpha  \alpha^{*}  \alpha^{1*} :     \\
&\quad + 2  \delta_{r,0}P^2   \alpha^{1*}   \lambda^2   + 2 \delta_{r,0}P  \partial P \alpha^{1*} \lambda    +    \delta_{r,0}P \partial ^2P \alpha^{1*}      \\
\end{align*}

Next we have 
\begin{align*}
 2  [P:\alpha^1&  \partial \alpha^{1*}  : _{\lambda} P :\alpha \alpha^*\alpha^{1*} : ]   =2  P^2: \alpha \alpha^*    \partial \alpha^{1*}  : ,
\end{align*}
and 
\begin{align*}
\frac{1}{2}  [\partial P:\alpha^1&    \alpha^{1*}  : _{\lambda} 2P :\alpha \alpha^*\alpha^{1*} : ]   =P\partial P:     \alpha \alpha^*\alpha^{1*}    : .
\end{align*}
 
 Next we have 
\begin{align*}
  [P:\alpha&  \partial \alpha^{*}  : {_\lambda} \beta^1\alpha^*]= P\beta^1  \partial \alpha^{*} ,
\end{align*}
\begin{align*}
  [P:\alpha^1&  \partial \alpha^{1*}  : {_\lambda} P\beta\alpha^{1*}]= P^2 \beta \partial \alpha^{1*}  
\end{align*}
and 
\begin{align*}
\frac{1}{2}  [\partial P:\alpha^1&   \alpha^{1*}  : {_\lambda} P\beta\alpha^{1*}]= \frac{1}{2}P\partial P \beta   \alpha^{1*} .
\end{align*}

Next we compute
\begin{align*}
\chi_0P(z)&:\alpha^1(z)\partial_z\alpha^{1*}(z):P(w)\partial_w\alpha^{1*}(w)  \\
&\sim \chi_0P(z)\partial_z\alpha^{1*}(z)\lfloor \alpha^1(z),\partial_w\alpha^{1*}(w)\rfloor P(w)  \\
&\sim \chi_0(P(w)+(z-w)\partial_wP(w))\iota_{z,w}\left(\frac{1}{(z-w)^2}\right)\left(\partial_w\alpha^{1*}(w)+(z-w)\partial_w^2\alpha^{1*}(w)\right)P(w)  \\ 
&\sim \chi_0(P(w)+(z-w)\partial_wP(w))\iota_{z,w}\left(\frac{1}{(z-w)^2}\right)\left(\partial_w\alpha^{1*}(w)+(z-w)\partial_w^2\alpha^{1*}(w)\right)P(w) ,
\end{align*}
so that 
\begin{align*}
\chi_0[P:\alpha^1\partial \alpha^{1*}:{_\lambda}P\partial \alpha^{1*}]=\chi_0P\left(P\partial \alpha^{1*}\lambda +\partial P\partial \alpha^{1*}+  P\partial^2 \alpha^{1*}\right)
\end{align*}
and \begin{align*}
\frac{1}{2}\chi_0P(z)&:\alpha^1(z)\partial_z\alpha^{1*}(z):\partial_w P(w)\alpha^{1*}(w)  \\
&\sim \frac{1}{2}\chi_0P(z)\partial_z\alpha^{1*}(z)\lfloor \alpha^1(z),\alpha^{1*}(w)\rfloor P(w)  \\
&\sim \frac{1}{2}\chi_0P(w) \iota_{z,w}\left(\frac{1}{(z-w)}\right)\alpha^{1*}(w)\partial_wP(w)  \\ 
\end{align*}
so that 
\begin{align*}
\frac{1}{2}\chi_0[P:\alpha^1\partial \alpha^{1*}:{_\lambda}\partial P\alpha^{1*}]=\frac{1}{2}\chi_0P\partial P\partial \alpha^{1*}.
\end{align*}

We also have 
\begin{align*}
\frac{1}{2}\partial_z P(z):&\alpha^1(z)\alpha^{1*}(z):\chi_0P(w)\partial_w\alpha^{1*}(w) \\
&\sim \frac{1}{2}\chi_0\partial_z P(z)\alpha^{1*}(z)\lfloor \alpha^{1}(z)\partial_w\alpha^{1*}(w) \rfloor P(w) \\
&\sim \frac{1}{2}\chi_0\left(\partial_w P(w)+(z-w)\partial_w^2 P(w)\right)\left(\alpha^{1*}(w)+(z-w)\partial_w\alpha^{1*}(z)\right)P(w)\iota_{z,w}\left(\frac{1}{(z-w)^2}\right) \\
\end{align*}
so that 
\begin{align*}
\frac{1}{2}\chi_0[\partial  P:\alpha^1\alpha^{1*}:{_\lambda}P\partial \alpha^{1*}] = \frac{1}{2}\chi_0P\left(\partial  P\alpha^{1*}\lambda+\partial^2 P\alpha^{1*}+\partial P\partial \alpha^{1*}\right).
\end{align*}
Next we have 
\begin{align*}
\frac{1}{4}\partial_z P(z):&\alpha^1(z)\alpha^{1*}(z):\chi_0\partial_wP(w)\alpha^{1*}(w) \\
&\sim \frac{1}{4}\chi_0\partial_z P(z)\alpha^{1*}(z)\lfloor \alpha^{1}(z)\alpha^{1*}(w) \rfloor \partial_wP(w) \\
&\sim \frac{1}{4}\chi_0\partial_w P(w) \alpha^{1*}(w)P(w)\iota_{z,w}\left(\frac{1}{(z-w)}\right) \\
\end{align*}
so that 
\begin{align*}
\frac{1}{4}[\partial P:\alpha^1\alpha^{1*}:{_\lambda}\chi_0\partial P\alpha^{1*}] = \frac{1}{4}\chi_0(\partial P)^2 \alpha^{1*}.
\end{align*}

Next on the agenda is 
\begin{align*}
-\nu^2\kappa_0P(z)&:\beta(z)^2:\beta^1(w)\alpha^*(w) \\
&\sim -2\nu^2\kappa_0P(z)\beta(z)\lfloor \beta(z)\beta^1(w)\rfloor\alpha^*(w) \\
&\sim 4\nu^2\kappa_0 \chi_1\left(P(w)+(z-w)\partial_wP(w))\right)\left(\beta(w)+(z-w)\partial_w\beta(w)\right)\iota_{z,w}\left(\frac{1}{(z-w)^2}\right)\alpha^*(w)  \\
\end{align*}
so that 
\begin{align*}
-\nu^2\kappa_0[P:\beta^2:{_\lambda}\beta^1\alpha^*] =4\nu^2\kappa_0\chi_1 \left(P\beta\lambda+\partial P \beta+P\partial\beta\right).
\end{align*}
We have
\begin{align*}
-\nu^2\kappa_0 &:\beta^1(z)^2:\beta^1(w)\alpha^*(w) \\
&\sim -2\nu^2\kappa_0\beta^1(z)\lfloor \beta^1(z)\beta^1(w)\rfloor\alpha^*(w) \\
&\sim 4\nu^2\kappa_0^2 \left(\beta^1(w)+(z-w)\partial_w\beta^1(w)\right)\iota_{z,w}\left(P(w)\frac{1}{(z-w)^2}+\frac{1}{2}\partial P(w)\frac{1}{z-w}\right)\alpha^*(w)  \\
\end{align*}
so that 
\begin{align*}
-\nu^2\kappa_0^2 [:(\beta^1)^2:{_\lambda}\beta^1\alpha^*]= 4\nu^2\kappa_0 ^2\alpha^*\left(P\beta^1\lambda+  P\partial\beta^1 +\frac{1}{2}\partial P\beta^1\right) \\
\end{align*}
Second to last we have 
\begin{align*}
- \nu^2\kappa_0&P(z):\beta(z)^2    P(w)\beta(w)\alpha^{1*}(w)  \\
&\sim  4\nu^2\kappa_0^2P(z)\beta(z)   P(w)\iota_{z,w}\left(\frac{1}{(z-w)^2}\right)\alpha^{1*}(w)   \\
&\sim 4\nu^2\kappa_0^2\left(P(w)+(z-w)\partial_wP(w)\right)\left(\beta(w) +(z-w)\partial_w\beta(w)\right)  P(w)\iota_{z,w}\left(\frac{1}{(z-w)^2}\right)\alpha^{1*}(w),
\end{align*}
so that
\begin{align*}
- \nu^2\kappa_0&[P:\beta^2 {_\lambda} P\beta\alpha^{1*}]  = 4\nu^2\kappa_0^2\alpha^{1*}P\left(P\beta\lambda+\partial P\beta+P\partial \beta\right).
\end{align*}

And finally we have
\begin{align*}
-\nu^2\kappa_0 \Big[&:(\beta^1)^2:{_\lambda}\left(P\beta \alpha^{1*}\right) \Big] = 4\nu^2\kappa_0\chi_1P\alpha^{1*}\beta(\beta^1\lambda +\partial \beta^1).
\end{align*}

Putting this altogether with $\chi_1=0$ we get 
\begin{align*}
[\pi(\bar{\mathbf d})_\lambda \tau(  e^1)]
&=\Big[\left(P\left(:\alpha\partial\alpha^*: +:\alpha^1\partial\alpha^{1*}:\right )
+\frac{1}{2}\partial P:\alpha^1\alpha^{1*}:\right) _\lambda    \\
&\qquad\Big(\alpha^1(\alpha^*)^2
 	+P\left(\alpha^1 (\alpha^{1*} )^2 +2 : \alpha \alpha^{*} \alpha^{1*}:\right)  \\
&\qquad +\beta^1 \alpha^* +P\beta \alpha^{1*} +\chi_0\left(P  \partial \alpha^{1*}   +\frac{1}{2}\partial P  \alpha^{1*} \right)\Big)  \Big]\\
&\qquad - \Big[\left(\nu^2P\kappa_0:\beta^2:+\nu^2\kappa_0:(\beta^1)^2:\right)_\lambda\left(  \beta^1 \alpha^* +P\beta \alpha^{1*}\right) \Big] \\ \\ 
&=2 P:  \alpha^1\alpha^*\partial\alpha^*: +P :\alpha^1 (\alpha^*)^2 :\lambda+\partial P:\alpha^1 (\alpha^*)^2 : + P:\partial \alpha^1  (\alpha^*)^2 : 
\\
&\qquad  - \frac{1}{2}\partial P:\alpha^1 (\alpha^*)^2 : +P^2  :\alpha^1(\alpha^{1*})^2 :\lambda   +  P\partial P  :\alpha^1(\alpha^{1*})^2 : \\ 
&\qquad +  P^2 :\partial \alpha^1(\alpha^{1*})^2 :   +2P^2  \alpha^1\alpha^{1*} \partial \alpha^{1*}: + 2  P^2  \delta_{r,0} \alpha^{1*}   \lambda^2   + 2 \delta_{r,0}P  \partial P \alpha^{1*} \lambda    +    \delta_{r,0}P \partial ^2P \alpha^{1*}  \\
&\qquad +\frac{1}{2}P\partial P :\alpha^1(\alpha^{1*})^2 :  - \delta_{r,0} P\partial P \alpha^{1*}\lambda - \delta_{r,0} P\partial ^2P \alpha^{1*}
\\
&\qquad +2  P^2   :\alpha  \partial\alpha^{*}  \alpha^{1*} :   + 2  P^2  :\alpha  \alpha^{*}  \alpha^{1*} : \lambda    \\
&\qquad + 2P \partial P  :\alpha  \alpha^{*}  \alpha^{1*} :     + 2  P^2   :\partial \alpha  \alpha^{*}  \alpha^{1*} :    + 2  P^2  \delta_{r,0} \alpha^{1*}   \lambda^2   + 2 \delta_{r,0}P  \partial P \alpha^{1*} \lambda    +    \delta_{r,0}P \partial ^2P \alpha^{1*}      \\ 
&\qquad+ 2  P^2: \alpha \alpha^*    \partial \alpha^{1*}  : +P\partial P:     \alpha \alpha^*\alpha^{1*}    : \\
&\qquad  + P\beta^1  \partial \alpha^{*} 
+ P^2 \beta \partial \alpha^{1*}  
+\frac{1}{2}P\partial P \beta   \alpha^{1*}  \\
&\qquad +\chi_0P\left(P\partial \alpha^{1*}\lambda +\partial P\partial \alpha^{1*}+  P\partial^2 \alpha^{1*}\right) +\frac{1}{2}\chi_0P\partial P\partial \alpha^{1*} \\
&\qquad +\frac{1}{2}\chi_0P\left(\partial  P\alpha^{1*}\lambda+\partial^2 P\alpha^{1*}+\partial P\partial \alpha^{1*}\right) +\frac{1}{4}\chi_0(\partial P )^2\alpha^{1*}\\
&\qquad + 4\nu^2\kappa_0 ^2\alpha^*\left(P\beta^1\lambda+  P \partial \beta^1 +\frac{1}{2}\partial P\beta^1\right)   \\
&\qquad + 4\nu^2\kappa_0^2\alpha^{1*}P\left(P\beta\lambda+\partial P\beta+P\partial \beta\right) 
\end{align*}

\begin{align*}
&=2 P:  \alpha^1\alpha^*\partial\alpha^*: + P:\partial \alpha^1  (\alpha^*)^2 : +  P\partial P  :\alpha^1(\alpha^{1*})^2 :  \\
&\qquad 
+P :\alpha^1 (\alpha^*)^2 :\lambda+\partial P:\alpha^1 (\alpha^*)^2 : \\
&\qquad +  P^2 :\partial \alpha^1(\alpha^{1*})^2 :   +2P^2  \alpha^1\alpha^{1*} \partial \alpha^{1*}:+2  P^2   :\alpha  \partial\alpha^{*}  \alpha^{1*} :  - \frac{1}{2}\partial P:\alpha^1 (\alpha^*)^2 :  \\ 
&\qquad + 2P \partial P  :\alpha  \alpha^{*}  \alpha^{1*} :     + 2  P^2   :\partial \alpha  \alpha^{*}  \alpha^{1*} : +\frac{1}{2}P\partial P :\alpha^1(\alpha^{1*})^2 :  \\ 
&\qquad  +    \delta_{r,0}P \partial ^2P \alpha^{1*}   + P\beta^1  \partial \alpha^{*}  + P^2 \beta \partial \alpha^{1*}  
+\frac{1}{2}P\partial P \beta   \alpha^{1*}   \\ 
&\qquad+ 2  P^2: \alpha \alpha^*    \partial \alpha^{1*}  : +P\partial P:     \alpha \alpha^*\alpha^{1*}    :   +    \delta_{r,0}P \partial ^2P \alpha^{1*}- \delta_{r,0} P\partial ^2P \alpha^{1*} \\  \\
&\qquad  +P^2  :\alpha^1(\alpha^{1*})^2 :\lambda + 2  P^2  \delta_{r,0} \alpha^{1*}   \lambda^2   + 2 \delta_{r,0}P  \partial P \alpha^{1*} \lambda   \\
&\qquad  - \delta_{r,0} P\partial P \alpha^{1*}\lambda 
\\
&\qquad   + 2  P^2  :\alpha  \alpha^{*}  \alpha^{1*} : \lambda     + 2  P^2  \delta_{r,0} \alpha^{1*}   \lambda^2   + 2 \delta_{r,0}P  \partial P \alpha^{1*} \lambda       \\ 
&\qquad +\chi_0P^2\partial \alpha^{1*}\lambda + 4\nu^2\kappa_0 ^2\alpha^*P\beta^1\lambda+ 4\nu^2\kappa_0^2\alpha^{1*}P^2\beta\lambda\\ \\
&\qquad +\chi_0P\left(\partial P\partial \alpha^{1*}+  P\partial^2 \alpha^{1*}\right) +\frac{1}{2}\chi_0P\partial P\partial \alpha^{1*} \\
&\qquad +\frac{1}{2}\chi_0P\left(\partial  P\alpha^{1*}\lambda+\partial^2 P\alpha^{1*}+\partial P\partial \alpha^{1*}\right) +\frac{1}{4}\chi_0(\partial P)^2 \alpha^{1*}\\
&\qquad+ 4\nu^2\kappa_0 ^2\alpha^*\left( \partial P\beta^1+\frac{1}{2} \partial P\beta^1\right)   \\
&\qquad + 4\nu^2\kappa_0^2\alpha^{1*}P\left(\partial P\beta+P\partial \beta\right)  \\ \\
&=P\partial(:  \alpha^1(\alpha^*)^2: )  \\
&\qquad +  P\partial(P :\alpha^1(\alpha^{1*})^2 :+2  P   :\alpha   \alpha^{*}  \alpha^{1*} :)  - \frac{1}{2}\partial P:\alpha^1 (\alpha^*)^2 :  \\ 
&\qquad    -\frac{1}{2}P\partial P :\alpha^1(\alpha^{1*})^2 :  \\ 
&\qquad  +    \delta_{r,0}P \partial ^2P \alpha^{1*}   + P\beta^1  \partial \alpha^{*}  + P^2 \beta \partial \alpha^{1*}  
+\frac{1}{2}P\partial P \beta   \alpha^{1*}   \\ 
&\qquad  +    \delta_{r,0}P \partial ^2P \alpha^{1*}- \delta_{r,0} P\partial ^2P \alpha^{1*} \\  \\
&\qquad +P\left( :\alpha^1 (\alpha^*)^2 : +P\left(  :\alpha^1(\alpha^{1*})^2 :   + 2    :\alpha  \alpha^{*}  \alpha^{1*} :\right)+ \alpha^*\beta^1+  \alpha^{1*}P\beta\right) \lambda    \\
&\qquad + 4 P^2  \delta_{r,0} \alpha^{1*}   \lambda^2   + 2 \delta_{r,0}P  \partial P \alpha^{1*} \lambda   - \delta_{r,0} P\partial P \alpha^{1*}\lambda    + 2 \delta_{r,0}P  \partial P \alpha^{1*} \lambda       \\ 
&\qquad +\chi_0P^2\partial \alpha^{1*}\lambda +\frac{1}{2}\chi_0P\partial  P\alpha^{1*}\lambda\\ \\
&\qquad +  P\partial P  :\alpha^1(\alpha^{1*})^2 : +P\partial P:     \alpha \alpha^*\alpha^{1*}  +\partial P:\alpha^1 (\alpha^*)^2 : \\
&\qquad  +\chi_0P\left(\partial P\partial \alpha^{1*}+  P\partial^2 \alpha^{1*}\right) +\frac{1}{2}\chi_0P\partial P\partial \alpha^{1*} \\
&\qquad +\frac{1}{2}\chi_0P\left(\partial^2 P\alpha^{1*}+\partial P\partial \alpha^{1*}\right) +\frac{1}{4}\chi_0(\partial P)^2 \alpha^{1*}\\
&\qquad+ \alpha^*P\left( \partial P\beta^1 +\frac{1}{2} \partial P\beta^1\right)   \\
&\qquad + \alpha^{1*}P\left(\partial P\beta+P\partial \beta\right) 
\end{align*}

On the other hand 
\begin{align*}
[\pi(\bar{\mathbf d})_\lambda \tau(  e^1)]&=P\partial\tau( e^1) +P\tau(e^1)\lambda  +\frac{1}{2}\partial P\tau(e^1) \checkmark\\ 
&=P \partial\Big(\alpha^1(\alpha^*)^2
 	+P\left(\alpha^1 (\alpha^{1*} )^2 +2 : \alpha \alpha^{*} \alpha^{1*}:\right)  \\ 
&\qquad +\beta^1 \alpha^* +P\beta \alpha^{1*} +(\kappa_0+4\delta_{r,0})\left(P  \partial\alpha^{1*}   +\frac{1}{2}\partial P \alpha^{1*} \right) \Big) \\ \\
&\qquad +P  \Big(\alpha^1(\alpha^*)^2
 	+P\left(\alpha^1 (\alpha^{1*} )^2 +2 : \alpha \alpha^{*} \alpha^{1*}:\right)  \\
&\qquad +\beta^1 \alpha^* +P\beta \alpha^{1*} +(\kappa_0+4\delta_{r,0})\left(P  \partial\alpha^{1*}   +\frac{1}{2}\partial P \alpha^{1*} \right) \Big) \lambda\\ \\ 
&\qquad +\frac{1}{2}\partial P\Big(\alpha^1(\alpha^*)^2
 	+P\left(\alpha^1 (\alpha^{1*} )^2 +2 : \alpha \alpha^{*} \alpha^{1*}:\right)  \\
&\qquad +\beta^1 \alpha^* +P\beta \alpha^{1*} +(\kappa_0+4\delta_{r,0})\left(P  \partial\alpha^{1*}   +\frac{1}{2}\partial P \alpha^{1*} \right) \Big) \\ \\
&=P  \Big(\partial\alpha^1(\alpha^*)^2
 	+\partial P\left(\alpha^1 (\alpha^{1*} )^2 +2 : \alpha \alpha^{*} \alpha^{1*}:\right)  \\ 
&\qquad +\partial\beta^1 \alpha^* +\partial P\beta \alpha^{1*} +(\kappa_0+4\delta_{r,0})\left(\partial P  \partial\alpha^{1*}   +\frac{1}{2}\partial^2P \alpha^{1*} \right) \Big) \\ \\
&\quad +P  \Big(2 \alpha^1 \alpha^*\partial\alpha^*
 	+P\left(\partial\alpha^1 (\alpha^{1*} )^2 +2 : \partial\alpha \alpha^{*} \alpha^{1*}:\right)  \\ 
&\qquad +\beta^1\partial \alpha^* +P\partial \beta \alpha^{1*} +(\kappa_0+4\delta_{r,0})\left(P  \partial^2\alpha^{1*}   +\frac{1}{2}\partial P\partial \alpha^{1*} \right) \Big) \\ \\
&\quad   +  P^2\left(2\alpha^1 \alpha^{1*}\partial\alpha^{1*}+ 2(: \alpha\partial \alpha^{*} \alpha^{1*}:+ : \alpha \alpha^{*} \partial\alpha^{1*}:) +\beta\partial \alpha^{1*} \right)   \\ \\ 
&\quad   +P  \Big(\alpha^1(\alpha^*)^2
 	+P\left(\alpha^1 (\alpha^{1*} )^2 +2 : \alpha \alpha^{*} \alpha^{1*}:\right)  \\ 
&\qquad  +\beta^1 \alpha^* +P\beta \alpha^{1*} +(\kappa_0+4\delta_{r,0})\left(P  \partial\alpha^{1*}   +\frac{1}{2}\partial P \alpha^{1*} \right) \Big)\lambda \\ \\ 
&\qquad +\frac{1}{2}\partial P\Big(\alpha^1(\alpha^*)^2
 	+P\left(\alpha^1 (\alpha^{1*} )^2 +2 : \alpha \alpha^{*} \alpha^{1*}:\right)  \\
&\qquad +\beta^1 \alpha^* +P\beta \alpha^{1*} +(\kappa_0+4\delta_{r,0})\left(P  \partial\alpha^{1*}   +\frac{1}{2}\partial P \alpha^{1*} \right) \Big) \\  \\
\end{align*}

Thus we need to compare
\begin{align*}
&=P\partial( 2  P   :\alpha   \alpha^{*}  \alpha^{1*} :)   \\ 
&\qquad   +P\partial P:     \alpha \alpha^*\alpha^{1*}  
\end{align*}
with
\begin{align*}
&P  \Big( \partial P\left( 2 : \alpha \alpha^{*} \alpha^{1*}:\right)  \\ \\
& +P  \Big(P\left(  2 : \partial\alpha \alpha^{*} \alpha^{1*}:\right) \Big) \\ \\
&   +  P^2\left(  2(: \alpha\partial \alpha^{*} \alpha^{1*}:+ : \alpha \alpha^{*} \partial\alpha^{1*}:) \right)   \\ \\ 
&+\frac{1}{2}\partial P\Big(P\left( 2 : \alpha \alpha^{*} \alpha^{1*}:\right) 
\end{align*}
which are equal.

The last bracket for the gauge algebra is 
\begin{align*}
[\pi(\bar{\mathbf d}^1)_\lambda \tau(  e^1)]
&=\Big[\left(:\alpha^1\partial \alpha^*:+P:\alpha\partial \alpha^{1*}:+\frac{1}{2}\partial P
	:\alpha\alpha^{1*}:+\nu:\beta\beta^1:\right) _\lambda    \\
&\qquad\Big(:\alpha^1(\alpha^*)^2:
 	+P\left(:\alpha^1 (\alpha^{1*} )^2 :+2 : \alpha \alpha^{*} \alpha^{1*}:\right)  \\
&\qquad +\beta^1 \alpha^* +P\beta \alpha^{1*} +\chi_0\left(P  \partial \alpha^{1*}   +\frac{1}{2}\partial P  \alpha^{1*} \right)\Big)  \Big]\\ \\ 
&=\Big[\left(:\alpha^1\partial \alpha^*:+P:\alpha\partial \alpha^{1*}:+\frac{1}{2}\partial P
	:\alpha\alpha^{1*}: \right) _\lambda    \\
&\qquad\Big(:\alpha^1(\alpha^*)^2:
 	+P\left(:\alpha^1 (\alpha^{1*} )^2: +2 : \alpha \alpha^{*} \alpha^{1*}:\right)  \\
&\qquad +\beta^1 \alpha^* +P\beta \alpha^{1*} +\chi_0\left(P  \partial \alpha^{1*}   +\frac{1}{2}\partial P  \alpha^{1*} \right)\Big)  \Big]\\
&\quad +\nu\Big[ :\beta\beta^1:{_\lambda}   \Big(\beta^1 \alpha^* +P\beta \alpha^{1*} \Big)  \Big].
\end{align*}

We have twenty one brackets to compute when expanding the first summand above. 
 The first is $\Big[:\alpha^1\partial \alpha^*:{ _\lambda } \alpha^1(\alpha^*)^2  \Big]=0$.
The second is 
\begin{align*}
\Big[:\alpha^1\partial \alpha^*:{ _\lambda } P:\alpha^1 (\alpha^{1*} )^2: \Big]&=2P:\alpha^1 \alpha^{1*}  \partial \alpha^*:
\end{align*}
  The third is obtained from 
\begin{align*}
2 :\alpha^1(z)&\partial_z \alpha^*(z):P(w): \alpha(w) \alpha^{*}(w) \alpha^{1*}(w): \\
&\sim 2\lfloor \partial_z \alpha^*(z),\alpha(w)\rfloor  :\alpha^1(z)P(w)\alpha^{*}(w) \alpha^{1*}(w): \\
&\quad +2P(w)\lfloor \alpha^1(z)\alpha^{1*}(w)\rfloor: \partial_z \alpha^*(z)\alpha(w) \alpha^{*}(w):  \\
&\quad +2 P(w)\lfloor \partial_z \alpha^*(z),\alpha(w)\rfloor \lfloor\alpha^1(z)\alpha^{1*}(w)\rfloor \alpha^{*}(w)  \\ \\ 
&\sim 2\iota_{z,w}P(w)\left(\frac{1}{(z-w)^2}\right) :\alpha^1(z)\alpha^{*}(w) \alpha^{1*}(w): \\
&\quad +2\iota_{z,w}P(w)\left(\frac{1}{(z-w)}\right): \partial_z \alpha^*(z)\alpha(w) \alpha^{*}(w):  \\
&\quad +2\iota_{z,w}P(w)\left(\frac{1}{(z-w)^3}\right)\  \alpha^{*}(w).
\end{align*} 
Thus 
\begin{align*}
2 [:\alpha^1\partial\alpha^*:_{\lambda}:P \alpha \alpha^{*} \alpha^{1*}:]=2P\left(:\alpha^1\alpha^{*}\alpha^{1*}:\lambda +:\partial\alpha^1\alpha^{*}\alpha^{1*}:\right) +2P : \partial  \alpha^*\alpha\alpha^{*}:  + 2\delta_{r,0} P\alpha^{*}\lambda^{(2)}.
\end{align*}
The fourth calculation is $[:\alpha^1\partial \alpha^*:{ _\lambda}\beta^1 \alpha^*]=0$.  The fifth is 
obtained from 
\begin{align*}
:\alpha^1(z)\partial_z \alpha^*(z):&P(w)\beta(w)\alpha^{1*}(w) \\
&\sim 2P(w)\lfloor\alpha^1(z),\alpha^{1*}(w) \rfloor\partial_z \alpha^*(z)\beta(w),
\end{align*}
so that 
\begin{align*}
[:\alpha^1\partial \alpha^*:{_\lambda}P\beta\alpha^{1*}] = P \partial  \alpha^*\beta.
\end{align*}
The sixth and seventh calculation are
\begin{align*}
[:\alpha^1\partial \alpha^*:{_\lambda} \chi_0\left(P  \partial \alpha^{1*}   +\frac{1}{2}\partial P  \alpha^{1*} \right)]&=\chi_0P( \partial \alpha^*\lambda+\partial^2\alpha^*)   +\frac{1}{2}\chi_0\partial P\partial \alpha^* .
\end{align*}
The eighth calculated from by 
\begin{align*}
P(z):\alpha(z)&\partial_z\alpha^{1*}(z)::\alpha^1(w)(\alpha^*(w))^2: \\
&\sim P(z)\iota_{z,w}\left(\frac{1}{(z-w)^2}\right):\alpha(z)  (\alpha^*(w))^2: \\
&\quad + 2P(z)\iota_{z,w}\left(\frac{1}{(z-w)}\right):\partial_z\alpha^{1*}(z)\alpha^1(w)\alpha^*(w): \\
&\quad+  2P(z)\iota_{z,w}\left(\frac{1}{(z-w)}\right)\iota_{z,w}\left(\frac{1}{(z-w)^2}\right) \alpha^*(w)
\end{align*}
so that 
\begin{align*}
[P:\alpha\partial\alpha^{1*}{_\lambda}\alpha^1(\alpha^*)^2] 
&= P:\alpha  (\alpha^*)^2:\lambda+\partial P :\alpha  (\alpha^*)^2:+ P :\partial\alpha  (\alpha^*)^2: \\
&\quad + 2P  :\partial \alpha^{1*}\alpha^1\alpha^*:  +  2\delta_{r,0}(P\lambda^{(2)}+\partial P\lambda +\frac{1}{2}\partial^2P) \alpha^*
\end{align*}
The ninth bracket is calculated from 
\begin{align*}
P(z):\alpha(z)&\partial_z\alpha^{1*}(z): P:\alpha^1(w) (\alpha^{1*} (w))^2: \\
&\sim 2P(z)\lfloor\partial_z\alpha^{1*}(z)\alpha^1(w)\rfloor P:\alpha(z) (\alpha^{1*}(w) )^2: \\
&\sim 2P(z)\iota_{z,w}\left(\frac{1}{(z-w)^2}\right) P:\alpha(z) (\alpha^{1*}(w) )^2: \\
\end{align*}
which gives us
\begin{align*}
[P:\alpha&\partial\alpha^{1*}{_\lambda} P:\alpha^1 (\alpha^{1*} )^2]= P^2:\alpha (\alpha^{1*} )^2:  \lambda+P\partial P :\alpha (\alpha^{1*} )^2: +P^2 :\partial\alpha (\alpha^{1*} )^2: \\
\end{align*}
The tenth bracket is calculated from 
\begin{align*}
2[P:\alpha\partial \alpha^{1*}:{_\lambda} : P\alpha \alpha^{*}\alpha^{1*}:]= 2P^2:\alpha \alpha^{1*}\partial\alpha^{1*}:.
\end{align*}
The eleventh is calculated from 
\begin{align*}
[P:\alpha\partial \alpha^{1*}:{_\lambda} \beta^1 \alpha^*]=P \beta^1\partial \alpha^{1*} 
\end{align*}
The twelfth bracket is
\begin{align*}
[P:\alpha\partial \alpha^{1*}:{_\lambda} P\beta \alpha^{1*}]=0.
\end{align*}
The thirteenth and fourteenth brackets are  zero.

Next we have 
\begin{align*}
\frac{1}{2}[\partial P:\alpha  \alpha^{1*}:{_\lambda} :\alpha^1(\alpha^*)^2:]=-\frac{1}{2}\partial P:\alpha   (\alpha^*)^2:+\partial P:\alpha^1\alpha^{1*}\alpha^*:
\end{align*}
and the sixteenth bracket is 
\begin{align*}
\frac{1}{2}[\partial P:\alpha  \alpha^{1*}:{_\lambda} P:\alpha^1 (\alpha^{1*} )^2: ]=-\frac{1}{2}P\partial P:\alpha   (\alpha^{1*} )^2:.
\end{align*}
The seventeenth is 
\begin{align*}
\frac{1}{2}[\partial P:\alpha  \alpha^{1*}:{_\lambda}2 P: \alpha \alpha^{*} \alpha^{1*}:]=P\partial P :\alpha( \alpha^{1*})^2:
\end{align*}
The eighteenth is 
\begin{align*}
\frac{1}{2}[\partial P:\alpha  \alpha^{1*}:{_\lambda}\beta^1 \alpha^*]=\frac{1}{2}\partial P\beta^1\alpha^{1*}
\end{align*}
and the last three are zero.

In the second summation we have 
\begin{align*}
\nu\Big[ :\beta\beta^1:{_\lambda}   \Big(\beta^1 \alpha^* +P\beta \alpha^{1*} \Big)  \Big] =-2\nu \left(P\beta\lambda+P\partial \beta+ \frac{1}{2}\partial P\beta\right)\alpha^*\kappa_0 -2\nu\kappa_0 P\beta^1\alpha^{1*} \lambda-2\nu\kappa_0 P\partial\beta^1\alpha^{1*}  
\end{align*}
Thus we get 
\begin{align*}
[\pi(\bar{\mathbf d}^1)_\lambda \tau(  e^1)]
&=2P:\alpha^1 \alpha^{1*}  \partial \alpha^*:+2P\left(:\alpha^1\alpha^{*}\alpha^{1*}:\lambda +:\partial\alpha^1\alpha^{*}\alpha^{1*}:\right) +2P : \partial  \alpha^*\alpha\alpha^{*}:  \\
&\quad  +2P\delta_{r,0}  \alpha^{*}\lambda^2  + P \partial  \alpha^*\beta  \\
&\quad +\chi_0P( \partial \alpha^*\lambda+\partial^2\alpha^*)   +\frac{1}{2}\chi_0\partial P\partial \alpha^* \\
&\quad + P:\alpha  (\alpha^*)^2:\lambda+\partial P :\alpha  (\alpha^*)^2:+ P :\partial\alpha  (\alpha^*)^2: \\
&\quad + 2P  :\partial \alpha^{1*}\alpha^1\alpha^*:  +  2\delta_{r,0} (P\lambda^{(2)}+\partial P\lambda +\frac{1}{2}\partial^2P) \alpha^* \\
&\quad +P^2:\alpha (\alpha^{1*} )^2:  \lambda+P\partial P :\alpha (\alpha^{1*} )^2: +P^2 :\partial\alpha (\alpha^{1*} )^2: \\
&\quad +2P^2:\alpha \alpha^{1*}\partial\alpha^{1*}: \\
&\quad +P \beta^1\partial \alpha^{1*} \\
&\quad-\frac{1}{2}\partial P:\alpha   (\alpha^*)^2:+\partial P:\alpha^1\alpha^{1*}\alpha^*:
 \\
&\quad -\frac{1}{2}P\partial P:\alpha   (\alpha^{1*} )^2: \\
&\quad +P\partial P :\alpha( \alpha^{1*})^2: \\
&\quad+\frac{1}{2}\partial P\beta^1\alpha^{1*} \\
&\quad -2\nu \left(P\beta\lambda+P\partial \beta+ \frac{1}{2}\partial P\beta\right)\alpha^*\chi_0 -2\nu P\beta^1\alpha^{1*} \lambda  -2\nu\kappa_0 P\partial\beta^1\alpha^{1*}\\ \\
&=P\partial\Big(2:\alpha^1 \alpha^{1*}   \alpha^*:+:   \alpha^*\alpha\alpha^{*}:  +P :\alpha (\alpha^{1*} )^2:\Big)\\ 
&\quad + P \Big(\partial  \alpha^*\beta -2\nu \kappa_0\partial \beta\alpha^*  + \beta^1\partial \alpha^{1*}   -2\nu\kappa_0 \partial\beta^1\alpha^{1*} +\chi_0\partial^2\alpha^* \Big)  \\  \\
&\quad  +P\Big(:\alpha  (\alpha^*)^2:+2 :\alpha^1\alpha^{*}\alpha^{1*}:  +P:\alpha (\alpha^{1*} )^2:   -2\nu\kappa_0 \beta\alpha^*  -2\nu\kappa_0 \beta^1\alpha^{1*}   +\chi_0 \partial \alpha^*  \Big)\lambda\\ \\
&\quad +\frac{1}{2}\partial P \Big(: \alpha   (\alpha^*)^2:   + P :\alpha   (\alpha^{1*} )^2:  +2:\alpha^1\alpha^{1*}\alpha^*:
+  \beta^1\alpha^{1*} -2\nu\kappa_0\beta\alpha^*  + \chi_0 \partial \alpha^*\Big)
\end{align*}
On the other hand 
\begin{align*}
[\pi(\bar{\mathbf d}^1)_\lambda \tau(  e^1)]&=P\partial\tau( e) +P\tau(e)\lambda  +
\frac{1}{2}\partial P\tau(e) \\
&=P\partial\Big(:\alpha(z)(\alpha^*)^2 :+P:\alpha(\alpha^{1*})^2: +2 :\alpha^1\alpha^*\alpha^{1*}: \\
&\qquad +\beta\alpha^*+\beta^1\alpha^{1*} +\chi_0\partial\alpha^*  \Big)  \\
&\quad +P\Big(:\alpha(\alpha^*)^2 :+P:\alpha(\alpha^{1*})^2: +2 :\alpha^1\alpha^*\alpha^{1*}: \\
&\qquad +\beta\alpha^*+\beta^1\alpha^{1*} +\chi_0\partial\alpha^*  \Big)\lambda \\
&\quad +\frac{1}{2}\partial P\Big(\alpha(\alpha^*)^2 :+P:\alpha(\alpha^{1*})^2: +2 :\alpha^1\alpha^*\alpha^{1*}: \\
&\qquad +\beta\alpha^*+\beta^1\alpha^{1*} +\chi_0\partial\alpha^*  \Big).
\end{align*}
so provided $-2\nu\kappa_0=1$ we have the two are equal.
\end{proof}
\color{black}

\section{Further Comments}
 We plan to use the above construction to help elucidate the structure of these representations of a three point algebra, describe the space of their intertwining operators and eventually describe the center of a certain completion of the universal enveloping algebra for the three point algebra.  Also since in the representations above, part of the center acts nontrivially, so we will explore in future work whether there is an ambient gerbe lurking in the background. 



\begin{thebibliography}{CGLZ14}

\bibitem[BCF09]{MR2541818}
Andr{\'e} Bueno, Ben Cox, and Vyacheslav Futorny.
\newblock Free field realizations of the elliptic affine {L}ie algebra
  {$\mathfrak{sl}(2,{\bf R})\oplus(\Omega_R/d{\rm R})$}.
\newblock {\em J. Geom. Phys.}, 59(9):1258--1270, 2009.

\bibitem[Bre94a]{MR1261553}
Murray Bremner.
\newblock Generalized affine {K}ac-{M}oody {L}ie algebras over localizations of
  the polynomial ring in one variable.
\newblock {\em Canad. Math. Bull.}, 37(1):21--28, 1994.

\bibitem[Bre94b]{MR1303073}
Murray Bremner.
\newblock Universal central extensions of elliptic affine {L}ie algebras.
\newblock {\em J. Math. Phys.}, 35(12):6685--6692, 1994.

\bibitem[Bre95]{MR1249871}
Murray Bremner.
\newblock Four-point affine {L}ie algebras.
\newblock {\em Proc. Amer. Math. Soc.}, 123(7):1981--1989, 1995.

\bibitem[BT07]{MR2286073}
Georgia Benkart and Paul Terwilliger.
\newblock The universal central extension of the three-point {$\germ{sl}_2$}
  loop algebra.
\newblock {\em Proc. Amer. Math. Soc.}, 135(6):1659--1668 (electronic), 2007.

\bibitem[CF06]{MR2271362}
Ben~L. Cox and Vyacheslav Futorny.
\newblock Structure of intermediate {W}akimoto modules.
\newblock {\em J. Algebra}, 306(2):682--702, 2006.

\bibitem[CGLZ14]{MR3211093}
Ben Cox, Xiangqian Guo, Rencai Lu, and Kaiming Zhao.
\newblock {$n$}-point {V}irasoro algebras and their modules of densities.
\newblock {\em Commun. Contemp. Math.}, 16(3):1350047, 27, 2014.

\bibitem[CJ14]{MR3245847}
Ben Cox and Elizabeth Jurisich.
\newblock Realizations of the three-point {L}ie algebra {$\germ{sl}(2,{\mathcal
  R})\bigoplus(\Omega_{{\mathcal R}}/d{\mathcal R})$}.
\newblock {\em Pacific J. Math.}, 270(1):27--48, 2014.

\bibitem[CJM16]{MR3478523}
Ben Cox, Liz Jurisich, and Renato~A Martins.
\newblock The 3-point virasoro algebra and its action on a fock space.
\newblock {\em J. Math. Phys.}, 57(3):031702, 2016.

\bibitem[Cox08]{MR2373448}
Ben Cox.
\newblock Realizations of the four point affine {L}ie algebra
  {$\mathfrak{sl}(2,R)\oplus(\Omega_R/dR)$}.
\newblock {\em Pacific J. Math.}, 234(2):261--289, 2008.

\bibitem[EFK98]{MR1629472}
Pavel~I. Etingof, Igor~B. Frenkel, and Alexander~A. Kirillov, Jr.
\newblock {\em Lectures on representation theory and {K}nizhnik-{Z}amolodchikov
  equations}, volume~58 of {\em Mathematical Surveys and Monographs}.
\newblock American Mathematical Society, Providence, RI, 1998.

\bibitem[FBZ01]{MR1849359}
Edward Frenkel and David Ben-Zvi.
\newblock {\em Vertex algebras and algebraic curves}, volume~88 of {\em
  Mathematical Surveys and Monographs}.
\newblock American Mathematical Society, Providence, RI, 2001.

\bibitem[FF90]{MR92f:17026}
Boris~L. Fe{\u\i}gin and Edward~V. Frenkel.
\newblock Affine {K}ac-{M}oody algebras and semi-infinite flag manifolds.
\newblock {\em Comm. Math. Phys.}, 128(1):161--189, 1990.

\bibitem[FF99]{MR1729358}
Boris Feigin and Edward Frenkel.
\newblock Integrable hierarchies and {W}akimoto modules.
\newblock In {\em Differential topology, infinite-dimensional {L}ie algebras,
  and applications}, volume 194 of {\em Amer. Math. Soc. Transl. Ser. 2}, pages
  27--60. Amer. Math. Soc., Providence, RI, 1999.

\bibitem[Fre05]{MR2146349}
Edward Frenkel.
\newblock Wakimoto modules, opers and the center at the critical level.
\newblock {\em Adv. Math.}, 195(2):297--404, 2005.

\bibitem[Fre07]{MR2332156}
Edward Frenkel.
\newblock {\em Langlands correspondence for loop groups}, volume 103 of {\em
  Cambridge Studies in Advanced Mathematics}.
\newblock Cambridge University Press, Cambridge, 2007.

\bibitem[FS05]{MR2183958}
Alice Fialowski and Martin Schlichenmaier.
\newblock Global geometric deformations of current algebras as
  {K}richever-{N}ovikov type algebras.
\newblock {\em Comm. Math. Phys.}, 260(3):579--612, 2005.

\bibitem[FS06]{FailS}
Alice Fialowski and Martin Schlichenmaier.
\newblock Global geometric deformations of the virasoro algebra, current and
  affine algebras by krichever-novikov type algebra.
\newblock {\em math.QA/0610851}, 2006.

\bibitem[JK85]{JK}
H.~P. Jakobsen and V.~G. Kac.
\newblock A new class of unitarizable highest weight representations of
  infinite-dimensional {L}ie algebras.
\newblock In {\em Nonlinear equations in classical and quantum field theory
  (Meudon/Paris, 1983/1984)}, pages 1--20. Springer, Berlin, 1985.

\bibitem[Kas84]{MR772062}
Christian Kassel.
\newblock K\"ahler differentials and coverings of complex simple {L}ie algebras
  extended over a commutative algebra.
\newblock In {\em Proceedings of the {L}uminy conference on algebraic
  {$K$}-theory ({L}uminy, 1983)}, volume~34, pages 265--275, 1984.

\bibitem[KL82]{MR694130}
C.~Kassel and J.-L. Loday.
\newblock Extensions centrales d'alg\`ebres de {L}ie.
\newblock {\em Ann. Inst. Fourier (Grenoble)}, 32(4):119--142 (1983), 1982.

\bibitem[KL91]{MR1104840}
David Kazhdan and George Lusztig.
\newblock Affine {L}ie algebras and quantum groups.
\newblock {\em Internat. Math. Res. Notices}, (2):21--29, 1991.

\bibitem[KL93]{MR1186962}
D.~Kazhdan and G.~Lusztig.
\newblock Tensor structures arising from affine {L}ie algebras. {I}, {II}.
\newblock {\em J. Amer. Math. Soc.}, 6(4):905--947, 949--1011, 1993.

\bibitem[KN87a]{MR925072}
Igor~Moiseevich Krichever and S.~P. Novikov.
\newblock Algebras of {V}irasoro type, {R}iemann surfaces and strings in
  {M}inkowski space.
\newblock {\em Funktsional. Anal. i Prilozhen.}, 21(4):47--61, 96, 1987.

\bibitem[KN87b]{MR902293}
Igor~Moiseevich Krichever and S.~P. Novikov.
\newblock Algebras of {V}irasoro type, {R}iemann surfaces and the structures of
  soliton theory.
\newblock {\em Funktsional. Anal. i Prilozhen.}, 21(2):46--63, 1987.

\bibitem[KN89]{MR998426}
Igor~Moiseevich Krichever and S.~P. Novikov.
\newblock Algebras of {V}irasoro type, the energy-momentum tensor, and operator
  expansions on {R}iemann surfaces.
\newblock {\em Funktsional. Anal. i Prilozhen.}, 23(1):24--40, 1989.

\bibitem[Sch03a]{MR2058804}
Martin Schlichenmaier.
\newblock Higher genus affine algebras of {K}richever-{N}ovikov type.
\newblock {\em Mosc. Math. J.}, 3(4):1395--1427, 2003.

\bibitem[Sch03b]{MR1989644}
Martin Schlichenmaier.
\newblock Local cocycles and central extensions for multipoint algebras of
  {K}richever-{N}ovikov type.
\newblock {\em J. Reine Angew. Math.}, 559:53--94, 2003.

\bibitem[Ser77]{MR0450380}
Jean-Pierre Serre.
\newblock {\em Linear representations of finite groups}.
\newblock Springer-Verlag, New York-Heidelberg, 1977.
\newblock Translated from the second French edition by Leonard L. Scott,
  Graduate Texts in Mathematics, Vol. 42.

\bibitem[She03]{MR2072650}
O.~K. She{\u\i}nman.
\newblock Second-order {C}asimirs for the affine {K}richever-{N}ovikov algebras
  {$\widehat{\mathfrak g\mathfrak l}_{g,2}$} and {$\widehat{\mathfrak
  s\mathfrak l}_{g,2}$}.
\newblock In {\em Fundamental mathematics today ({R}ussian)}, pages 372--404.
  Nezavis. Mosk. Univ., Moscow, 2003.

\bibitem[She05]{MR2152962}
O.~K. Sheinman.
\newblock Highest-weight representations of {K}richever-{N}ovikov algebras and
  integrable systems.
\newblock {\em Uspekhi Mat. Nauk}, 60(2(362)):177--178, 2005.

\bibitem[SS98]{MR1666274}
M.~Schlichenmaier and O.~K. Scheinman.
\newblock The {S}ugawara construction and {C}asimir operators for
  {K}richever-{N}ovikov algebras.
\newblock {\em J. Math. Sci. (New York)}, 92(2):3807--3834, 1998.
\newblock Complex analysis and representation theory, 1.

\bibitem[SS99]{MR1706819}
M.~Shlichenmaier and O.~K. Sheinman.
\newblock The {W}ess-{Z}umino-{W}itten-{N}ovikov theory,
  {K}nizhnik-{Z}amolodchikov equations, and {K}richever-{N}ovikov algebras.
\newblock {\em Uspekhi Mat. Nauk}, 54(1(325)):213--250, 1999.

\bibitem[SV90]{MR1077959}
V.~V. Schechtman and A.~N. Varchenko.
\newblock Hypergeometric solutions of {K}nizhnik-{Z}amolodchikov equations.
\newblock {\em Lett. Math. Phys.}, 20(4):279--283, 1990.

\bibitem[Wak86]{W}
Minoru Wakimoto.
\newblock Fock representations of the affine {L}ie algebra ${A}\sp {(1)}\sb 1$.
\newblock {\em Comm. Math. Phys.}, 104(4):605--609, 1986.

\end{thebibliography}
\def\cprime{$'$} \def\cprime{$'$} \def\cprime{$'$} \def\cprime{$'$}
  \def\cprime{$'$} \def\cprime{$'$} \def\cprime{$'$} \def\cprime{$'$}


 \end{document}